\documentclass{siamltex}
\usepackage{amssymb,amsmath,amscd,verbatim}
\usepackage{enumerate}
\usepackage{algorithmic}
\usepackage{tikz}
\usetikzlibrary{shapes}
\usepackage{color}

\usepackage{caption}
\usepackage{subfigure}
\usepackage{multirow}
\usepackage{algorithmic}
\usepackage{algorithm}
\newcommand{\norm}[1]{\|#1\|_2}

\newcommand{\linspace}[2][ ]{\mathbb{#2}^{#1}}
\usepackage[pdftex,colorlinks=true,linkcolor=red!50!black,citecolor=blue!50!black]{hyperref}

\newcommand{\coarsevar}{\mathcal{C}}
\newcommand{\finevar}{\mathcal{F}}
\newcommand{\tvV}{\mathcal{V}}
\newcommand{\tvU}{\mathcal{U}}

\author{J.~Brannick\footnotemark[3] and K.~Kahl\footnotemark[4]}

\begin{document}

\title{Bootstrap Algebraic Multigrid for the 2D Wilson Dirac system}
\date{\today}
\renewcommand{\thefootnote}{\fnsymbol{footnote}}
\maketitle


\begin{AM}
65F10, 65N55, 65F30
\end{AM}

\begin{abstract}
We develop an algebraic multigrid method for solving the non-Hermitian
Wilson discretization of the 2-dimensional Dirac equation.
The proposed approach uses a bootstrap
setup algorithm based on a multigrid eigensolver. It computes test
vectors which define the least squares interpolation operators by
working mainly on coarse grids, leading to an efficient and
integrated self learning process for defining algebraic multigrid interpolation.  The
algorithm is motivated by the $\gamma_5$-symmetry of the Dirac
equation, which carries over to the Wilson discretization. This
discrete $\gamma_5$-symmetry is used to reduce a general Petrov Galerkin bootstrap
setup algorithm 
to a Galerkin method for the Hermitian and indefinite formulation of the Wilson matrix.
Kaczmarz relaxation is used as the multigrid smoothing scheme in both
the setup and solve phases of the resulting Galerkin algorithm.  The overall method is
applied to the odd-even reduced Wilson matrix, 
which also fulfills the discrete $\gamma_5$-symmetry. 
Extensive numerical results are presented to motivate the design and demonstrate the effectiveness
of the proposed approach.
\end{abstract}

\footnotetext[3]{Department of Mathematics,
Pennsylvania State University, University Park, PA 16802, USA (brannick@psu.edu). \thanks{Brannick's work was supported by the National Science Foundation under grants OCI-0749202 and DMS-810982.}}
\footnotetext[4]{Fachbereich Mathematik und Naturwissenschaften,
Bergische Universit\"at Wuppertal, D-42097 Wuppertal, Germany, 
(kkahl@math.uni-wuppertal.de). \thanks{Kahl's work was supported by the
  Deutsche Forschungsgemeinschaft through the Collaborative Research
  Centre SFB-TR 55 ``Hadron Physics from Lattice QCD''}}
\renewcommand{\thefootnote}{\arabic{footnote}}

\section{Introduction}\label{sec:intro}
Lattice quantum chromodynamics (QCD) is a numerical approach for
computing observables of quarks,  elementary particles, in cases where
perturbative methods diverge, see~\cite{DeGrand:2006zz} for an overview.  
Simulations of quarks require
approximating the QCD path integral using Monte Carlo methods, which 
involves generating discrete realizations of the gauge fields and,
then, computing observables by averaging over  these ensembles of
configurations.  In both of these stages of a lattice QCD calculation, the
discretized Dirac equation
\begin{equation}\label{eq:dirac}
  D \psi = b
\end{equation}
needs to be solved for numerous realizations of the gauge fields and, then, 
multiple right hand sides for each configuration.  
In this paper, we consider Wilson's discretization~\cite{PhysRevD.10.2445} of the Dirac
equation so that $D = D_0 +mI$, where $D_0$ denotes the non-Hermitian
mass-less Wilson matrix and the shift $m$ is related to the mass of
the quarks. We refer to $D$ as simply the Wilson matrix implying that a
particular shift $m$ is associated with it.

All existing lattice QCD algorithms suffer from what is referred to as {\em critical slowing down}, which is a direct result of 
the structure of the Wilson matrix. Specifically, as the shift, $m$, approaches 
physically relevant values the minimal eigenvalues
of $D$ approach zero linearly, which leads to a highly ill-conditioned system of equations and 
the stalling convergence 
of standard Krylov subspace methods
when applied to this system. As a result, the overall simulation becomes too costly at light masses
and up until now this has led to the use of non-physical heavy quark masses in lattice QCD simulations.
This, in turn, has motivated the extensive research 
that has been dedicated to the development of suitable multigrid preconditioners for
discretizations of the Dirac equation 
over the past three decades, see~\cite{Brower:1990ac,Brower:1991en,Brower:1991xv,Brower:1987dd,Brower:1990at,Brower:1991ni,Harmatz1991102}.  

The task of designing effective multigrid preconditioners for the
Wilson matrix 
is further complicated by the fact that the near kernel modes, i.e.,
the vectors $x$ such that $Dx \approx 0$ or $D^Hx \approx 0$, are
locally non-smooth. To be more precise, their entries depend significantly
on the specific values of the
given gauge field configuration, and, 
a precise understanding of
this dependence is not well understood
theoretically, yet. 
As a result, while earlier efforts in designing multigrid methods for
discretizations of the Dirac equation did lead to marked improvements
in some cases, methods with the potential to effectively remove
critical slowing down in general have emerged only in the past few
years in the context of adaptive~\cite{MBrezina_RFalgout_SMacLachlan_TManteuffel_SMcCormick_JRuge_2003}
or bootstrap~\cite{BAMG2} algebraic multigrid (AMG).

The main new component of the adaptive and bootstrap AMG approaches is the idea to 
use the AMG hierarchy to expose prototype(s) of the near kernel (test vectors)
that are not effectively treated by the solver and, then, adapt the
coarse spaces to incorporate them. In \emph{adaptive
AMG}~\cite{MBrezina_2005,MBrezina_RFalgout_SMacLachlan_TManteuffel_SMcCormick_JRuge_2003},
the solver is applied to appropriately formulated homogeneous problems on different grids to
compute a single test vector, which is then used to update the
restriction, interpolation, and coarse-grid operators on all grids.
This gives a new solver which can be used in another adaptive
cycle.  The process is then repeated in a sequence of adaptive cycles
until an efficient solver has been constructed. In contrast,  \emph{bootstrap
AMG} uses relaxation and a multigrid (eigen)solver 
based on the emerging AMG hierarchy to compute a collection of test
vectors in each of the bootstrap cycles and, then, a local least
squares problem is formulated to define interpolation operators that
approximate these vectors collectively.

Adaptive and bootstrap AMG setup algorithms have been
developed for smoothed aggregation multigrid~\cite{MBrezina_RFalgout_SMacLachlan_TManteuffel_SMcCormick_JRuge_2003}, element-free AMG~\cite{Vassilevski_2005}, 
and classical AMG~\cite{BAMG2,BAMG2010,Brannick_Trace_06,MBrezina_2005}. Promising
results of two- and three-grid adaptive aggregation AMG
preconditioners for the Wilson and 
Wilson Clover discretizations of the Dirac equation are found in~\cite{Osborne_POS_2009,Osborne_POS_2010} and in 
these works it has been demonstrated that adaptive AMG techniques can be used to construct effective preconditioners for the Wilson matrix.  
Some progress on combining adaptive AMG and bootstrap AMG techniques to develop
preconditioners for the Wilson and Wilson Clover formulation has also
been made~\cite{Bran_PRL_2010,BAMG2010,Rottman_arxiv}.  
These developments have shown that the bootstrap AMG approach, when combined with adaptive AMG, has the potential to dramatically reduce the costs of the adaptive setup process.  

In this paper, we design and analyze a multigrid solver for the
Wilson matrix based on the bootstrap AMG framework.
The proposed approach builds on our work in~\cite{BAMG2010}, where we
developed bootstrap AMG for solving Hermitian
and positive definite linear systems of equations
\begin{equation*}\label{aa}
A u = f.
\end{equation*}
Of particular interest in this previous work was the development of a bootstrap 
setup algorithm for the gauge Laplacian system with an emphasis on highly disordered 
gauge fields.  Since the gauge Laplacian is Hermitian
and positive definite, a Galerkin scheme based on a variational principle is 
the natural approach for solving this problem.  
The associated two-grid method involves a stationary linear iterative 
method (smoother) applied to the fine-grid system,
and a coarse-grid correction: given an approximation $w\in \mathbb
C^n$, compute an update $v\in \mathbb C^n$ by \medskip
\begin{enumerate}
\item Pre-smoothing: $y = w+M(f-Aw)$,
\item Correction: $v = y + PA_c^{-1}P^{H}(f-Ay), \quad A_c = P^{H}AP$.
\end{enumerate}\medskip
Here, $M$ is the approximate inverse of $A$ that defines the smoother and
$P: \mathbb{C}^{n_c} \mapsto \mathbb{C}^n$ with $n_c < n$ is the
interpolation operator that maps information from the coarse to the
fine grid. In the variational Galerkin scheme, restriction is defined by
the conjugate transpose $P^{H}$ of $P$.
Thus the error propagation operator for a Galerkin
two-grid method with one pre-smoothing step is given by 
\begin{equation*}\label{2-level}
E_{G} = (I - PA_c^{-1}P^{H}A)(I-MA).
\end{equation*} 
A multigrid algorithm is then obtained by recursively solving the coarse-grid
error equation, involving $A_c$, using another two-grid method.

Generally, there are two AMG approaches for solving non-Hermitian
problems like \eqref{eq:dirac} involving the non-Hermitian Wilson matrix $D$,
Galerkin~\cite{BankDupont,Bramble_93} and Petrov Galerkin
\cite{Bran_PRL_2010,gesa,Sala:2008:NPS:1461600.1461602} 
methods. A Petrov Galerkin method differs from the above Galerkin
approach in that the restriction operator $R:\mathbb{C}^{n} \mapsto
\mathbb{C}^{n_c}$ is no longer chosen as $P^{H}$.
Consequently the coarse-grid operator is given
by $D_{c} = RDP$.  The error propagation operator of a Petrov Galerkin
method applied to $D$, with one pre-smoothing step is then
\begin{equation*}\label{2-levelpg}
  E_{PG} = (I - PD_{c}^{-1}RD)(I-MD).
\end{equation*}
In~\cite{gesa}, heuristic motivation and two-grid convergence theory
of a Petrov Galerkin AMG approach in which the coarse spaces are constructed 
to approximate left and right singular vectors with small singular values are developed
for non-Hermitian problems. The
basic result that motivates this approach is as follows.

Let $\sigma_1 \leq \ldots \leq \sigma_n$ be the singular values of the
matrix $D$, i.e., $D = U \Sigma V^H$, with $U$ and $V$ unitary, and $\Sigma =
\operatorname{diag}(\sigma_1,...,\sigma_n)$, then it follows that
$\sigma_1 \leq |\lambda| \leq \sigma_n, $ for any eigenvalue $\lambda$
of $D$. 
This suggests that the right and left near kernel vectors, i.e., $x$ and $y$ such
that $$\frac{\|Dx\|}{\|x\|} \approx \min_v \frac{\|Dv\|}{\|v\|} 
\quad \mbox{and} \quad \frac{\|D^Hy\|}{\|y\|} \approx \min_w
\frac{\|D^Hw\|}{\|w\|},$$ 
are dominated by singular vectors rather than eigenvectors. To be more
precise, let $W = V U^H$ and define the Hermitian positive definite
matrices $WD = (D^HD)^{\frac12}$ and $DW =  (DD^H)^{\frac12}$.
Then the original non-Hermitian system $D\psi = b$ can
be reformulated in two ways as an equivalent Hermitian system using
$WD$ or $DW$ as the system matrix. Now, the fact that the eigenvectors
corresponding to the minimal eigenvalues of $WD$ are the right
singular vectors corresponding to the minimal singular values of $D$
and those for $DW$ are the left singular vectors corresponding to
minimal singular values of $D$ and because $WD$ and $DW$ are Hermitian
positive definite, they can be used to derive an approximation
property for the original problem involving $D$, assuming that $R$
is based on left singular vectors and
$P$ is based on right singular vectors corresponding to
small singular values.  Two grid convergence then follows from this approximation property 
together with the use of a suitable smoother.

We use this same reasoning to motivate the design of the proposed 
multigrid bootstrap AMG setup algorithm for the Wilson system.
We begin by considering a general Petrov Galerkin bootstrap AMG setup
strategy for computing left and right singular vectors of 
$D$, implicitly based on the equivalent Hermitian and indefinite formulation of the
singular value decomposition (SVD)~\cite{golub_kahan_1965,lanczos_1961}. 
Then, we use the $\gamma_5$-symmetry that the Dirac matrix satisfies
to derive a  relation between the subspaces spanned by left and right singular
vectors with small singular values, as well as a relation between the
right singular vectors of $D$ and the eigenvectors of the Hermitian
and indefinite form of the Wilson matrix $Z=\Gamma_5 D$, 
where $\Gamma_5$ is a simple unitary matrix. 
Finally, using these observations 
and a structure preserving form of interpolation we are able to reduce
the general Petrov Galerkin approach 
to an equivalent Galerkin based coarsening scheme applied to $Z$.  
Further, because in our construction the resulting Galerkin coarse-grid operators also satisfy the
$\gamma_5$-symmetry on all grids, the equivalence 
of the proposed Petrov Galerkin setup process for $D$ and the Galerkin scheme for
$Z$ also holds on all grids.  

Since the Wilson discretization of the Dirac equation is formulated on
a structured grid, we exploit this structure in the implementation of
the proposed approach.  Instead of solving systems with the Dirac
matrix $D$ we solve linear systems involving the Schur complement resulting from odd-even (red-black) reduction of the Wilson 
matrix $D$ as considered in~\cite{Osborne_POS_2009,MILC}. We use a classical AMG~\cite{geo,oAMG} form of interpolation with full
coarsening (cf.~Section~\ref{sec:coarsening}) on all grids of the hierarchy. 
The non-zero entries of $P$ are chosen based on the structure as well
and interpolation weights are computed row-wise using a least squares 
interpolation approach~\cite{BAMG2010}. 
The setup we use in the
implementation of the proposed algorithm combines a bootstrap setup based on a multigrid
eigensolver with adaptive AMG cycles. 
We mention that such a bootstrap-adaptive setup process has previously been 
considered in the context of developing eigensolvers for computing state vectors in Markov chain
applications~\cite{brannick_markov_2011}.

We chose to use Kaczmarz relaxation both in the setup and solve phases
of the proposed algorithm, which gives a stationary smoother that is 
guaranteed to converge for the Wilson matrix. Using a convergent stationary smoother allows us to 
access and analyze the performance of the resulting bootstrap 
method for computing singular vector approximations in a systematic
way.  This is in contrast to the
aggregation-based solvers in~\cite{Bran_PRL_2010}
and~\cite{Osborne_POS_2009}, where Krylov methods are used as the
multigrid smoother, leading to a non-stationary multigrid iteration,
and in~\cite{Rottman_arxiv}, where a Schwarz alternating procedure is used as
the smoother, which has been observed to diverge for linear systems
with Wilson matrices. 

An outline of the remainder of this paper is as follows.  First, in
Section \ref{Wilson}, we introduce the Wilson discretization of the
Dirac equation and discuss some of the features of this problem that
make it difficult to solve using iterative methods. Then, in Section
\ref{BAMG}, we present a Galerkin coarsening algorithm that combines
the weighted least squares process for constructing interpolation with
bootstrap and adaptive techniques for computing test vectors used in
this construction.  
Section \ref{sec:numres} contains numerical results of the proposed
method applied to the Wilson discretization of the $2$-dimensional Dirac equation.
We end with concluding remarks in Section \ref{sec:conclusions}.

\section{The Wilson discretization of the Dirac equation}\label{Wilson}
In Lattice QCD, the Dirac equation is typically analyzed on a
hypercube with periodic (or anti-periodic) boundary conditions. A 
brief description of the Wilson discretization that is the focus
of this paper is given in this section.

\begin{figure}
\centering
\scalebox{.4}{\begin{picture}(0,0)%
\includegraphics{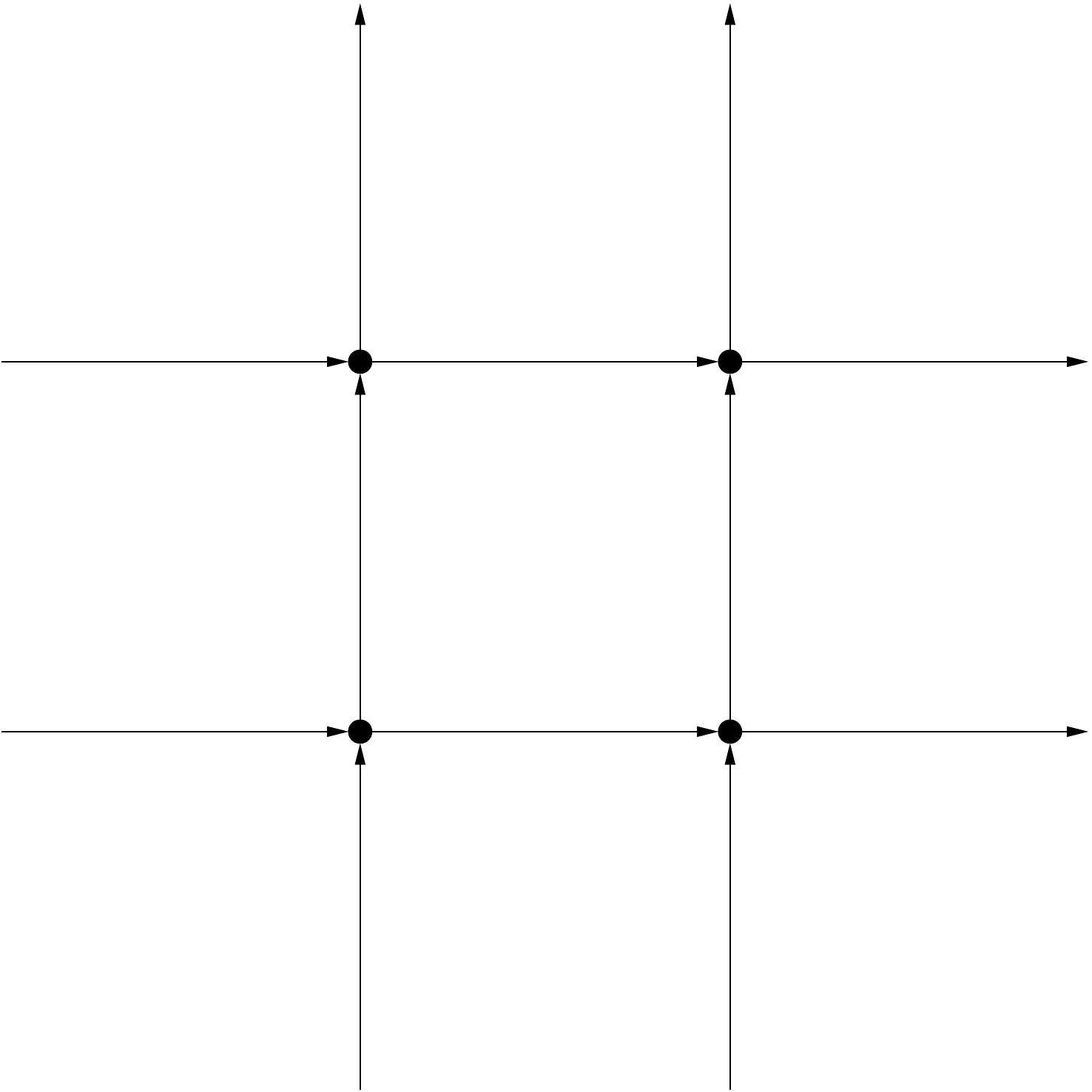}%
\end{picture}%
}
\setlength{\unitlength}{3947sp}%
\begingroup\makeatletter\ifx\SetFigFont\undefined%
\gdef\SetFigFont#1#2#3#4#5{%
  \reset@font\fontsize{#1}{#2pt}%
  \fontfamily{#3}\fontseries{#4}\fontshape{#5}%
  \selectfont}%
\fi\endgroup%
\scalebox{.4}{
\begin{picture}(7074,7074)(64,-6298)
\put(2476,-436){\makebox(0,0)[lb]{\smash{{\SetFigFont{20}{24.0}{\rmdefault}{\mddefault}{\updefault}{\color[rgb]{0,0,0}$U_{y}^{z+e_{y}}$}%
}}}}
\put(4876,-436){\makebox(0,0)[lb]{\smash{{\SetFigFont{20}{24.0}{\rmdefault}{\mddefault}{\updefault}{\color[rgb]{0,0,0}$U_{y}^{z+e_{x}+e_{y}}$}%
}}}}
\put(1051,-1411){\makebox(0,0)[lb]{\smash{{\SetFigFont{20}{24.0}{\rmdefault}{\mddefault}{\updefault}{\color[rgb]{0,0,0}$U_{x}^{z-e_{x}+e_{y}}$}%
}}}}
\put(3451,-1411){\makebox(0,0)[lb]{\smash{{\SetFigFont{20}{24.0}{\rmdefault}{\mddefault}{\updefault}{\color[rgb]{0,0,0}$U_{x}^{z+e_{y}}$}%
}}}}
\put(5851,-1411){\makebox(0,0)[lb]{\smash{{\SetFigFont{20}{24.0}{\rmdefault}{\mddefault}{\updefault}{\color[rgb]{0,0,0}$U_{x}^{z+e_{x}+e_{y}}$}%
}}}}
\put(2476,-1786){\makebox(0,0)[lb]{\smash{{\SetFigFont{20}{24.0}{\rmdefault}{\mddefault}{\updefault}{\color[rgb]{0,0,0}$z+e_{y}$}%
}}}}
\put(4876,-1786){\makebox(0,0)[lb]{\smash{{\SetFigFont{20}{24.0}{\rmdefault}{\mddefault}{\updefault}{\color[rgb]{0,0,0}$z+e_{x}+e_{y}$}%
}}}}
\put(2476,-2836){\makebox(0,0)[lb]{\smash{{\SetFigFont{20}{24.0}{\rmdefault}{\mddefault}{\updefault}{\color[rgb]{0,0,0}$U_{y}^{z}$}%
}}}}
\put(4876,-4186){\makebox(0,0)[lb]{\smash{{\SetFigFont{20}{24.0}{\rmdefault}{\mddefault}{\updefault}{\color[rgb]{0,0,0}$z+e_{x}$}%
}}}}
\put(2476,-4186){\makebox(0,0)[lb]{\smash{{\SetFigFont{20}{24.0}{\rmdefault}{\mddefault}{\updefault}{\color[rgb]{0,0,0}$z$}%
}}}}
\put(1051,-3811){\makebox(0,0)[lb]{\smash{{\SetFigFont{20}{24.0}{\rmdefault}{\mddefault}{\updefault}{\color[rgb]{0,0,0}$U_{x}^{z-e_{x}}$}%
}}}}
\put(3451,-3811){\makebox(0,0)[lb]{\smash{{\SetFigFont{20}{24.0}{\rmdefault}{\mddefault}{\updefault}{\color[rgb]{0,0,0}$U_{x}^{z}$}%
}}}}
\put(2476,-5236){\makebox(0,0)[lb]{\smash{{\SetFigFont{20}{24.0}{\rmdefault}{\mddefault}{\updefault}{\color[rgb]{0,0,0}$U_{y}^{z-e_{y}}$}%
}}}}
\put(4876,-5236){\makebox(0,0)[lb]{\smash{{\SetFigFont{20}{24.0}{\rmdefault}{\mddefault}{\updefault}{\color[rgb]{0,0,0}$U_{y}^{z+e_{x}-e_{y}}$}%
}}}}
\put(5851,-3811){\makebox(0,0)[lb]{\smash{{\SetFigFont{20}{24.0}{\rmdefault}{\mddefault}{\updefault}{\color[rgb]{0,0,0}$U_{x}^{z+e_{x}}$}%
}}}}
\put(4876,-2836){\makebox(0,0)[lb]{\smash{{\SetFigFont{20}{24.0}{\rmdefault}{\mddefault}{\updefault}{\color[rgb]{0,0,0}$U_{y}^{z+e_{x}}$}%
}}}}
\end{picture}}%
  \caption{Naming convention on the grid.\label{fig:boot:gaugeconfig}} 
\end{figure}
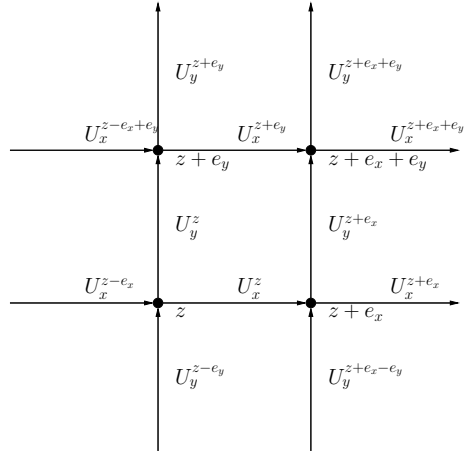  
Let $N_s$ denote the number of spin components and $N_c$ the
number of color components of the fields $\psi$, then the action of the
Wilson matrix $D$ on $\psi$ at a grid point $z\in\Omega = \{
1,...,N\}^d$, with $N$ the number of grid points in a given
space-time dimension, reads
\begin{equation}\label{eq:wilson}
 (D\psi)_z =   m_q \psi_z - \sum_{\mu=1}^d \bigg[\big(P_{\mu}^- \otimes  U^{z}_{\mu}\big) \psi_{z + e_\mu}+\big(P_{\mu}^+  \otimes \overline{U}^{z-e_\mu}_{\mu}\big) \psi_{z - e_\mu}\bigg],
 \end{equation}
where $m_q = m + 2(d-1)$ is a scalar quantity, $P_{\mu}^\pm = \frac{(I
  \pm \gamma_{\mu})}{2}$ satisfy $(P_{\mu}^\pm)^2 =
P_{\mu}^\pm,$ with $\gamma_\mu \in \linspace[N_{s}\times N_{s}]{C}$ denoting anti-commuting matrices, i.e.,
$\gamma_\mu\gamma_\nu + \gamma_\nu\gamma_\mu  = 2\delta_{\mu,\nu}I_{N_s}$, and $U_\mu^z
\in SU(N_c)$ are the discrete gauge fields belonging to the Lie group
$SU(N_c)$, $N_c \geq 1$, of 
$N_c \times N_c$ unitary matrices with $\operatorname{determinant}$
equal to one, and $e_{\mu}$ denotes the canoncical unit vector in the
$\mu$-direction, i.e., $z+e_{\mu}$ describes a shift from grid point
$z$ to its neighbor in the $\mu$-direction.
The unknown field, $\psi$, is defined at the grid points $z\in
\Omega$ with $N_{s}\cdot N_{c}$ variables per grid point. 
The discrete gauge fields $U_{\mu}^{z}$ are defined on the edges of the grid, 
as illustrated in Figure~\ref{fig:boot:gaugeconfig} in a 2-dimensional
setting.  The set $\{U_{\mu}^{z} \in SU(N_{c}), \mu = 1,\ldots,d, z\in \Omega\}$ of
discrete gauge fields $U_{\mu}^{z}$ is referred to as a gauge
configuration. For
a more detailed introduction to QCD and lattice QCD we refer 
to~\cite{DeGrand:2006zz,Gattringer:2010zz,montvay1994quantum}.

In the 2-dimensional setting that we consider in this paper, $N_s = 2$ and $N_c = 1$ so that the $\gamma_\mu$-matrices are
given by
\begin{eqnarray*}\label{eq:gamma}
 \gamma_1 = \begin{pmatrix}
  0    & 1   \\
   1   & 0 
\end{pmatrix}
\quad \mbox{and} \quad 
 \gamma_2 = \begin{pmatrix}
  0    & i   \\
   -i   & 0 
\end{pmatrix},
\end{eqnarray*}
and the fields $U_z^\mu$ which are defined on the edges of the grid  
belong to the $U(1)$ group, i.e., they are complex numbers with
modulus one.
In order to give an explicit expression for the action of $D$ on a field $\psi
\in \linspace[2N^2]{C}$ we consider a spin-permuted reordering. That
is, we write $\psi = (\psi_{1}^{T},\psi_2^{T})^{T}$ where $\psi_{1}$ represents
the variables with spin 1 at all grid points and $\psi_{2}$
represents the variables with spin 2. Ordering $D$ accordingly it
has the structure
\begin{equation}\label{eq:bst}
  D = \frac12\left(\begin{matrix} A & B \\ -B^{H} & A 
  \end{matrix}\right).
\end{equation} Here, the diagonal blocks, $A$, are referred to as
\textit{Gauge Laplacians}. They were introduced
originally by Wilson as a way to stabilize a covariant finite difference discretization of the Dirac
equation~\cite{PhysRevD.10.2445}.  
The action of $A \in \mathbb{C}^{N^2\times N^2}$ on a vector $\phi \in \linspace[{N^{2}}]{C}$ at a grid point $z \in \Omega$ reads
\begin{equation*}\label{eq:gauge_Laplace:def}
\left(A\phi\right)_{z} = (4+2m)\phi_{z} -  U^{z-e_{x}}_x\phi_{z-e_{x}} - U^{z-e_{y}}_{y}\phi_{z-e_{y}} -\overline{U}^{z}_{x}\phi_{z+e_{x}} -\overline{U}^{z}_{y}\phi_{z+e_{y}}, 
\end{equation*}
and the action of $B \in \mathbb{C}^{N^2\times N^2}$ on $\phi\in \linspace[{N^{2}}]{C}$ at site $z \in \Omega$ is given by
\begin{equation*}
  \left(B\phi\right)_{z} = \overline{U}^{z}_{x}\phi_{z+e_{x}} - U^{z-e_{x}}_{x}\phi_{z-e_{x}} + i\left(\overline{U}^{z}_{y}\phi_{z+e_{y}} -  U^{z-e_{y}}_{y}\phi_{z-e_{y}}\right).
\end{equation*} 
We note that if $U_\mu^z = 1$ for $\mu = x,y$ and all $z \in \Omega$, then $A$ is the standard 5-point Laplacian (plus a diagonal shift) and $B$ is a central difference approximation to the gradient of $\phi$.  

The gauge field configurations used in different calculations 
throughout the paper are
generated using a standard Metropolis algorithm
with a quenched Wilson gauge field action
\begin{equation}\label{eq:wilsongauge}
S = \sum_{z\in\Omega} \beta \mbox{Re} 
(   \overline{U}_y^{z} \overline{U}^{z+e_y}_{x} U^{z+e_x}_{y}  U_x^{z}).
\end{equation}
For details on the Metropolis algorithm and its implementation in this setting 
we refer to~\cite{DelDebbio:2005qa}.
In general, the distribution of the gauge fields depends on the parameter $\beta$
in \eqref{eq:wilsongauge}. The case $\beta \rightarrow \infty$ yields
$U_{\mu}^{z} \rightarrow 1$ for $\mu = x,y$ and all $z \in \Omega$. 
As $\beta \rightarrow 0$, the phases $\theta_{\mu}^{z}$ in
$U_{\mu}^{z} = e^{i\theta_{\mu}^{z}}$ become less correlated and the gauge fields 
become highly disordered, causing local oscillations in the near kernel components
of the Wilson matrix.  
In our tests, we consider three values of $\beta=3,6,10$ and nine configurations of the gauge fields
for each value, corresponding to steps 11,000, 12,000, ..., 19,000
of a standard Metropolis algorithm using the action given in~\eqref{eq:wilsongauge}.
We note that the same configurations are reused in all the tests.

\subsection{Singular vectors of the Wilson matrix}\label{Wilson:specs}
One difficulty that arises when designing multigrid solvers for the
Wilson discretization of the Dirac equation is that the support and local structure of the near kernel components of the Wilson matrix 
depend on the local values of the gauge fields.  
As an example, plots of the modulus 
of the individual spin components of the right
singular vectors to small singular
values of $D$ for $N=128$, $\beta=6$ at $\eta = 10^{-7}$ are provided in
Figures~\ref{fig:nk1},~\ref{fig:nk10} and~\ref{fig:nk20}. 
In this case, which is representative of what happens for physically relevant configurations, 
the singular vectors belonging to small singular values 
of the Wilson matrix are locally non-smooth. It is this feature 
that motivates the use of adaptive AMG techniques for the Wilson
matrix. 

In addition, we observe that some of the singular vectors to small singular
values are localized, e.g., the smallest in Figure~\ref{fig:nk1} and
$2$nd smallest in Figure~\ref{fig:nk10}. This is in contrast to the
singular vector belonging to the $10$th smallest singular value shown
in Figure~\ref{fig:nk20}. These findings of non-smooth, localized and
non-localized near kernel vectors is indicative of 
what occurs in practice. 
Related numerical studies on the local supports of the eigenvectors of the Wilson matrix are found in~\cite{local}.
Assuming a point-wise smoother, the coarse space basis used in the associated AMG solver 
for this problem must be able to approximate a large number of (possibly localized) near kernel vectors, 
which motivates the use of least squares interpolation as a technique for accurately approximating sets of such 
test vectors collectively.

\noindent\begin{minipage}[t]{\textwidth}
\begin{minipage}[t]{.3\textwidth}
\begin{figure}[H]
  \subfigure[First spin
  component.]{\includegraphics[width=\textwidth]{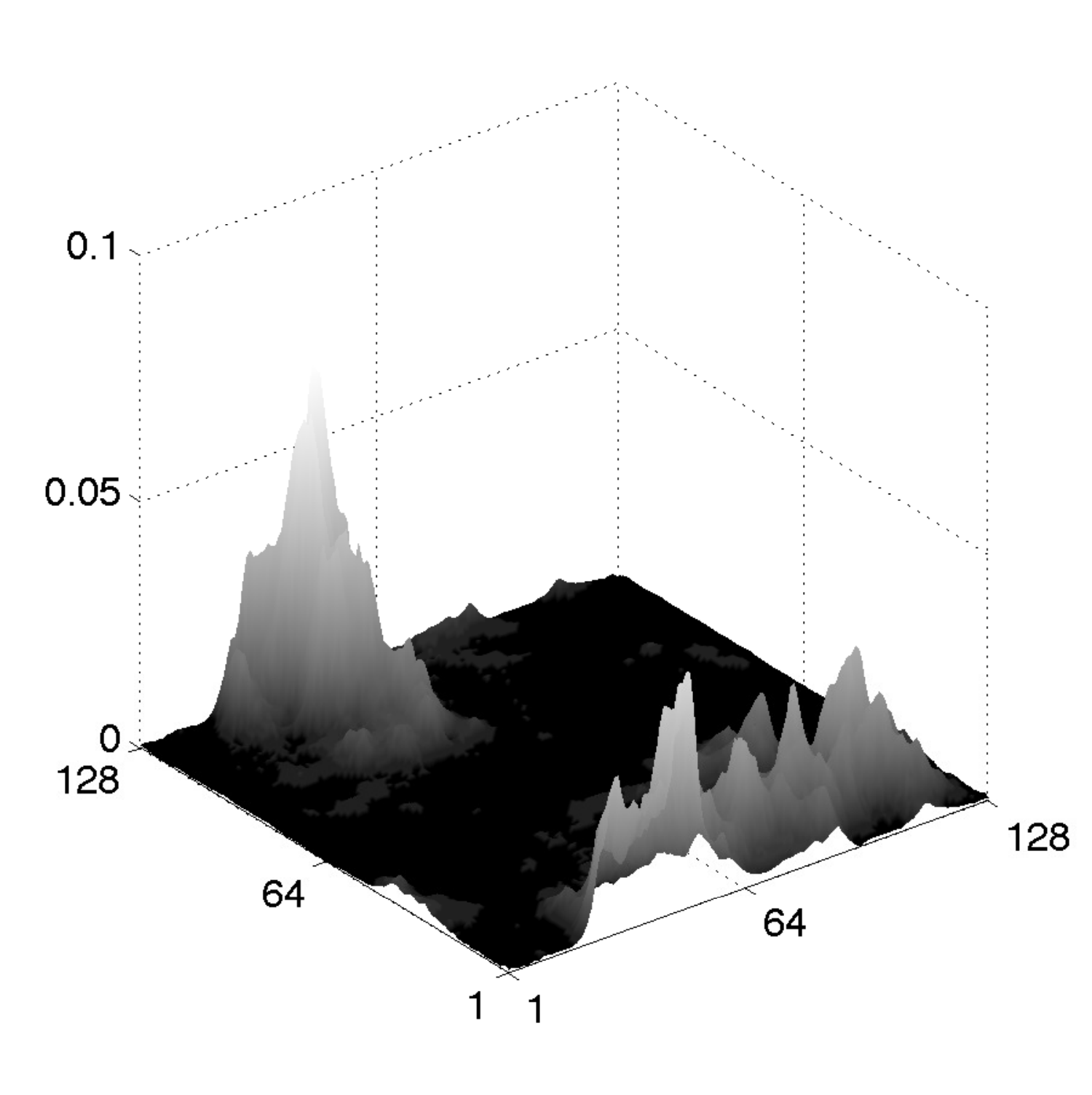}}
  \subfigure[Second spin
  component.]{\includegraphics[width=\textwidth]{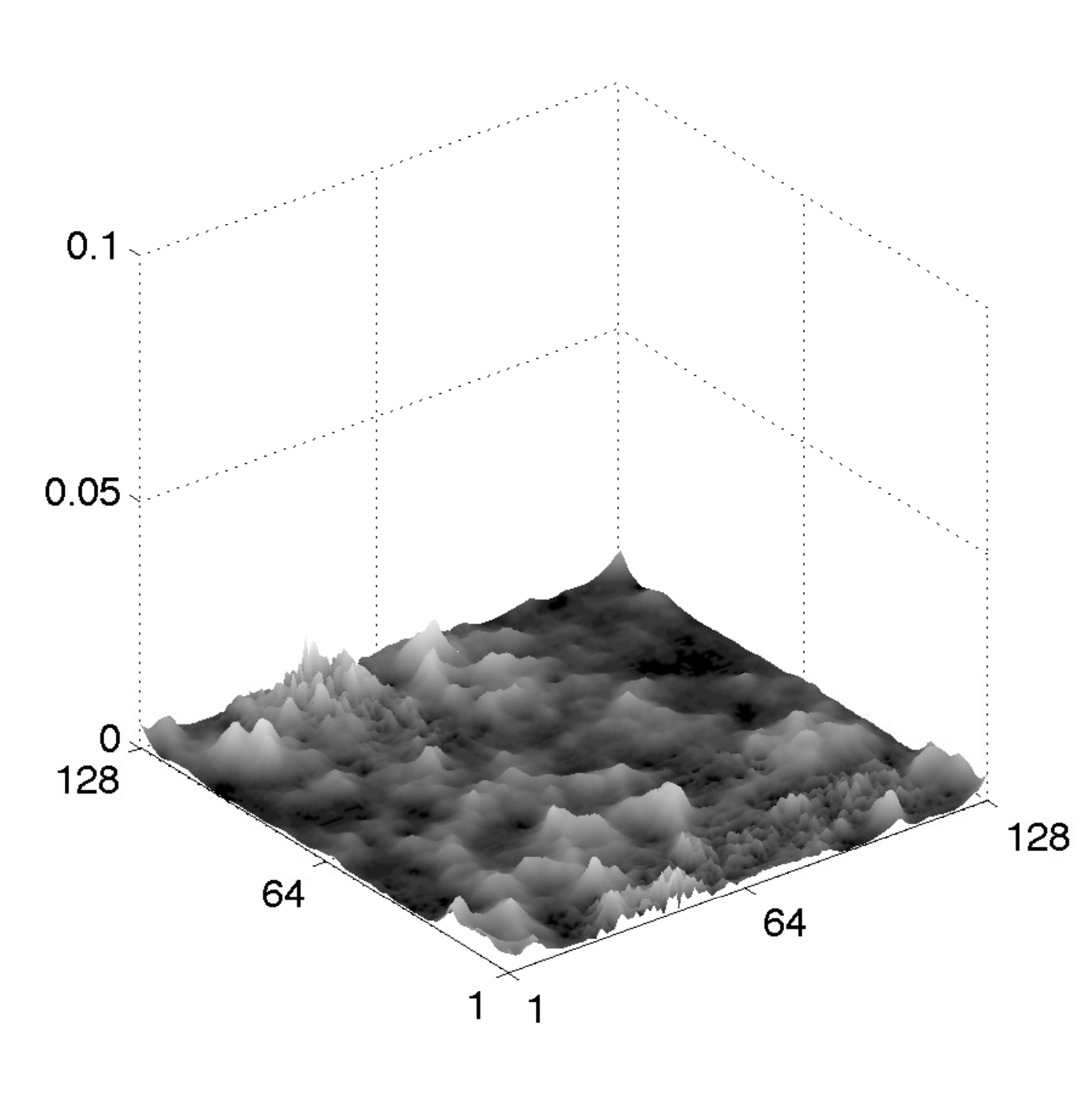}}
\caption{Modulus of the spin components of the right singular vector to the smallest
  singular value $\sigma_{1} = 8.49\cdot 10^{-8}$.\label{fig:nk1}}
\end{figure}
\end{minipage}
\hfill \vline \hfill
\begin{minipage}[t]{.3\textwidth}
\begin{figure}[H]
\subfigure[First spin component.]{\includegraphics[width=\textwidth]{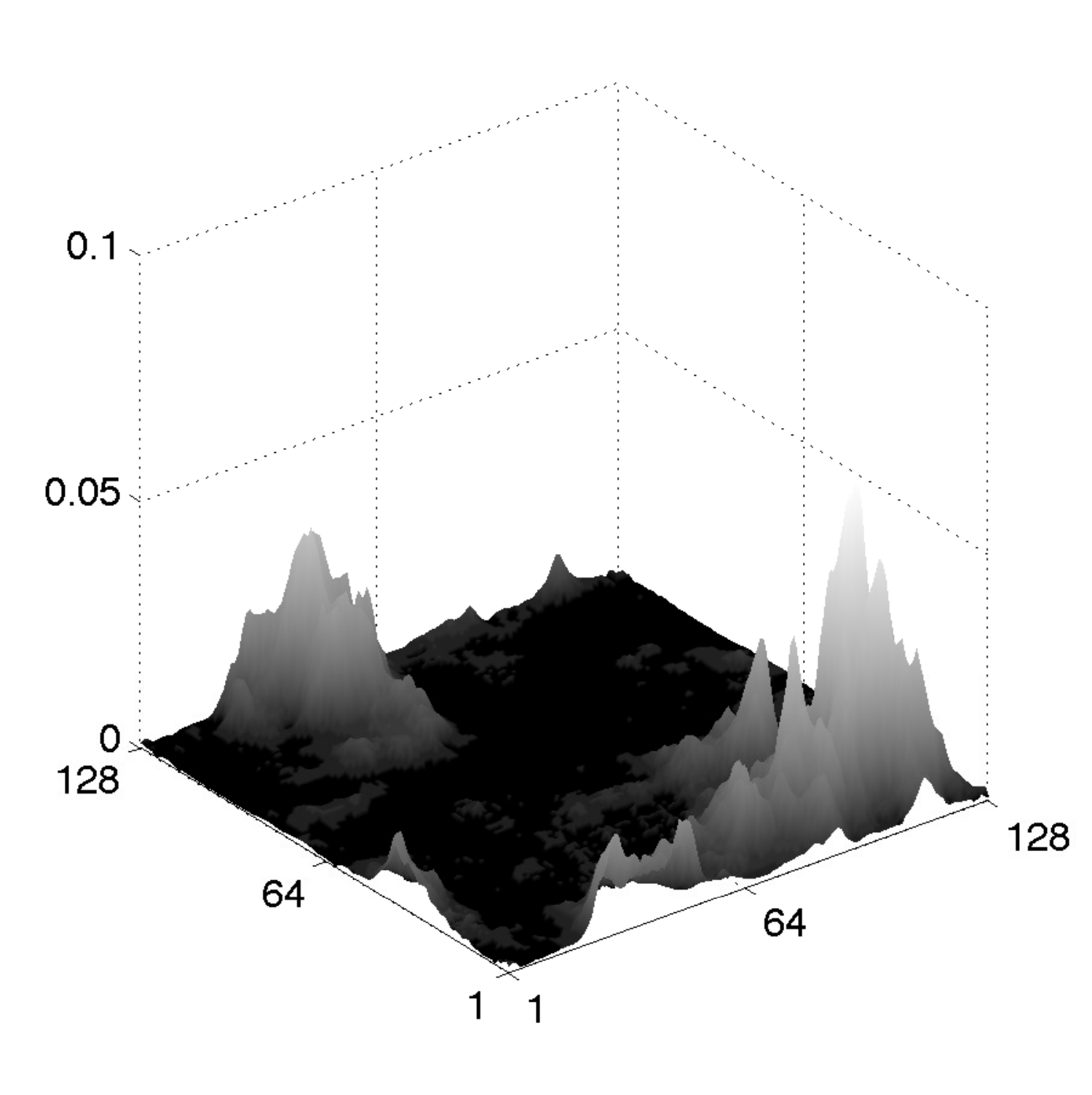}}
  \subfigure[Second spin
  component.]{\includegraphics[width=\textwidth]{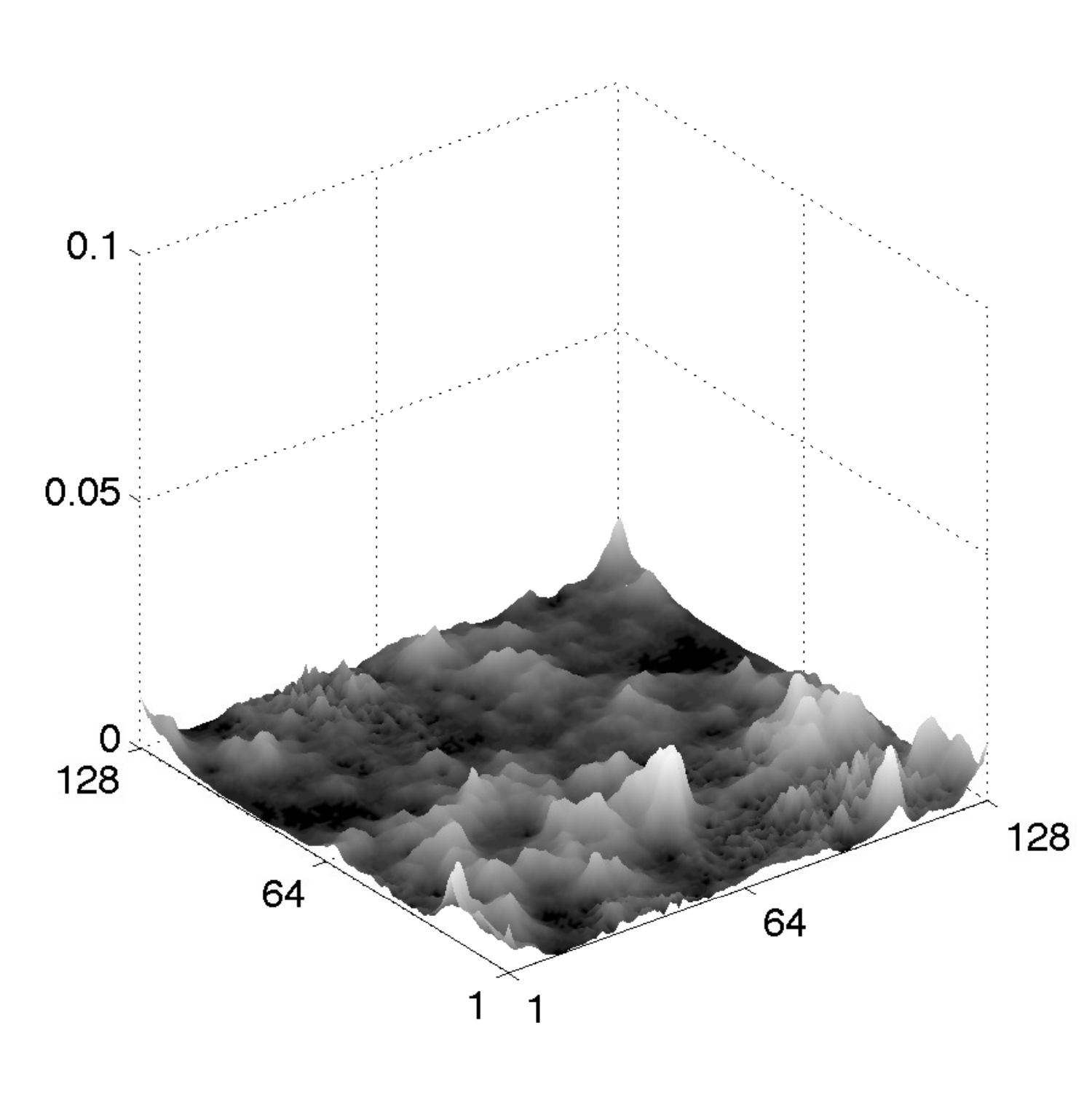}}
\caption{Modulus of the spin components of the right
  singular vector to the $2$nd smallest
  singular value $\sigma_{2} = 1.65\cdot 10^{-3}$.\label{fig:nk10}}
\end{figure}
\end{minipage}
\hfill \vline \hfill
\begin{minipage}[t]{.3\textwidth}
\begin{figure}[H]
  \subfigure[First spin component.]{\includegraphics[width=\textwidth]{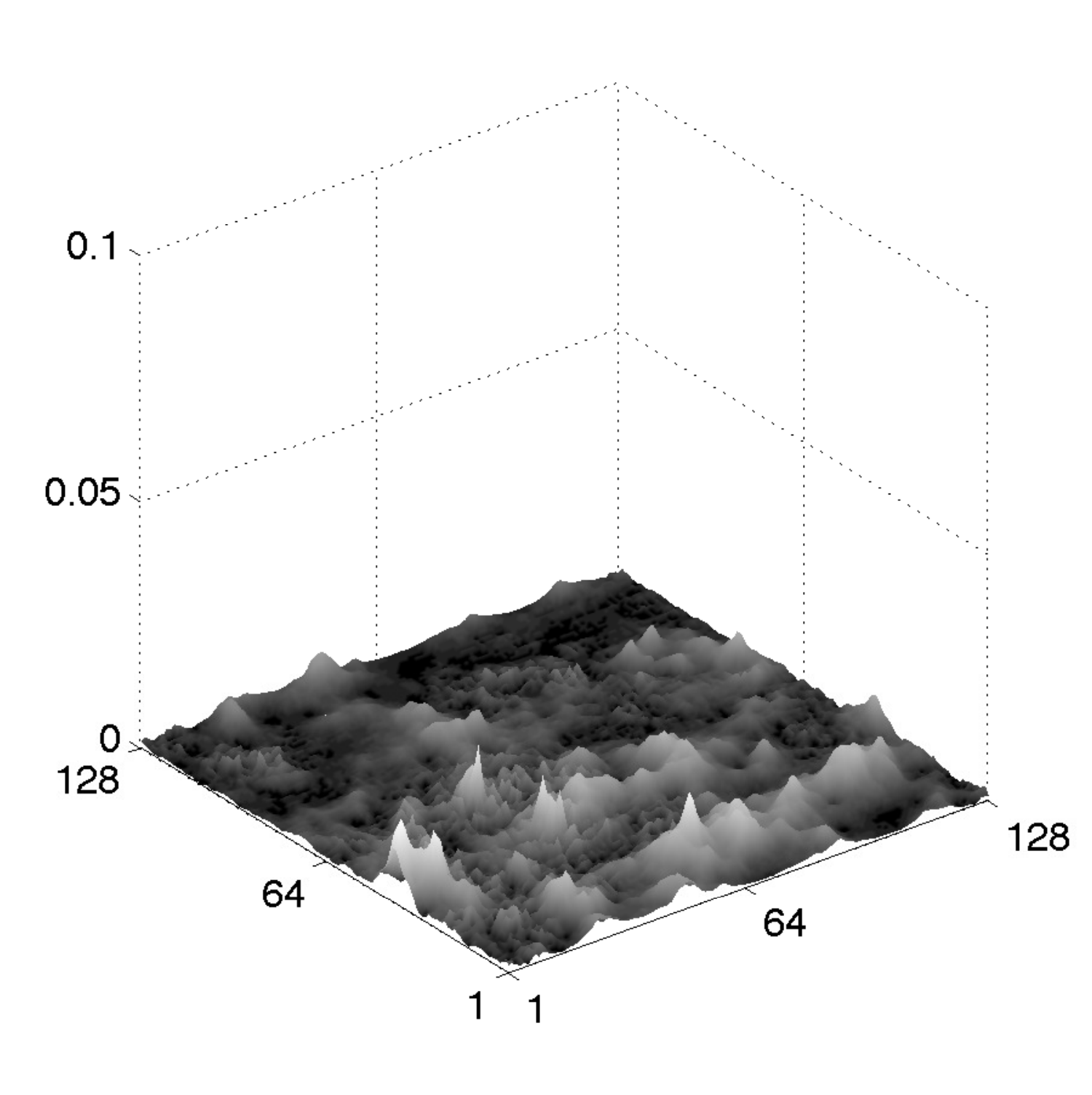}}
  \subfigure[Second spin component.]{\includegraphics[width=\textwidth]{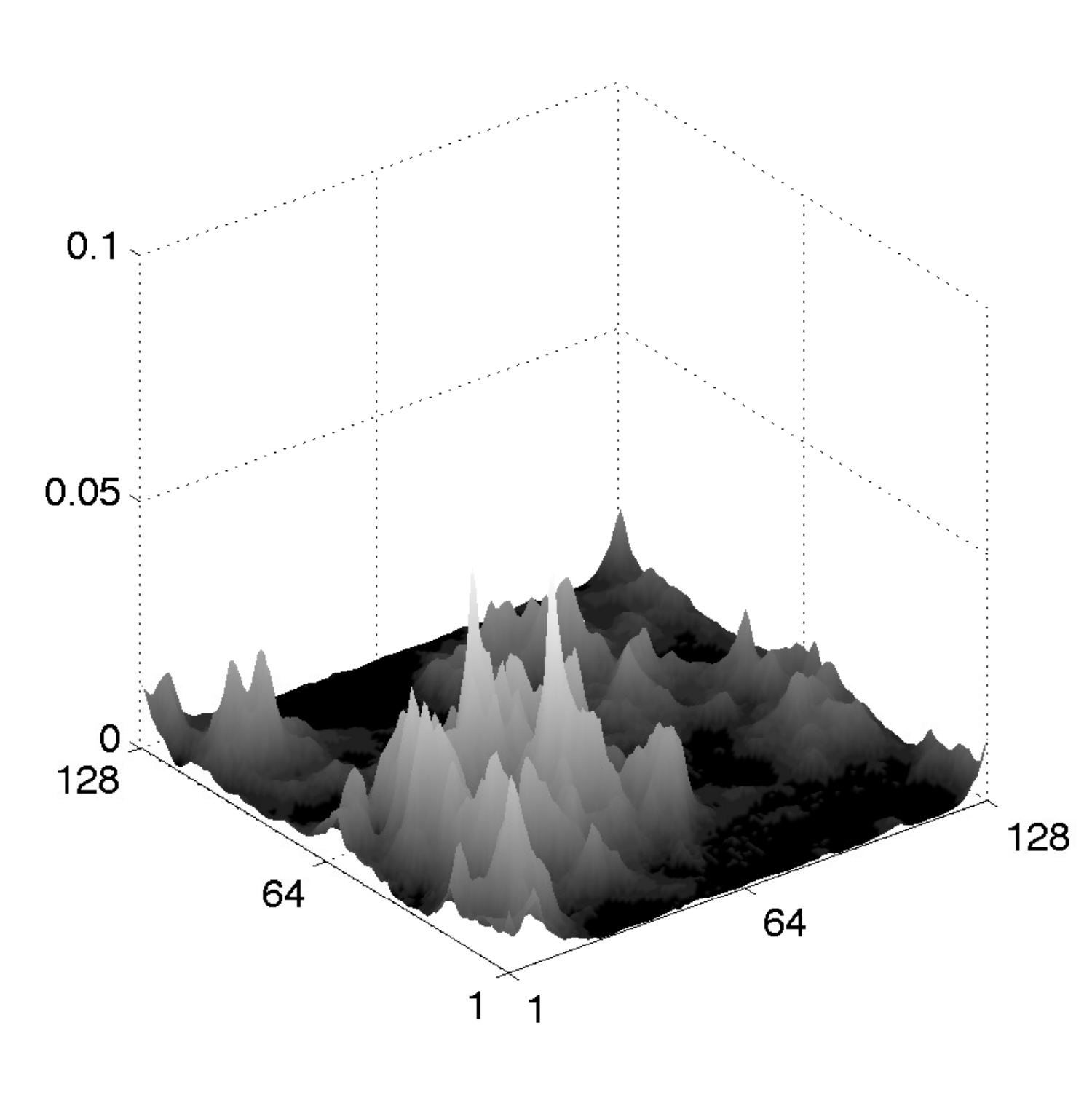}}
\caption{Modulus of the spin components of the right
  singular vector to the $10$th smallest
  singular value $\sigma_{10} = 2.09\cdot 10^{-2}$.\label{fig:nk20}}
\end{figure}
\end{minipage}
\end{minipage}\bigskip

\subsection{Spectrum of the Wilson matrix}\label{Wilson:spec}
Representing the near kernel vectors of the Wilson matrix $D$ in the
coarse space is further complicated by the fact that in practice 
 the shift $m$ is chosen such that 
\begin{equation}\label{eq:etamin}
\eta_{\min} (D) = \min_{\lambda \in \operatorname{spec}(D)} \mbox{Re} (\lambda)
\in \linspace{R},
\end{equation}
is positive and close to zero, with $\operatorname{spec}(D)$ denoting the spectrum of $D$.
The spectra of the mass-less Wilson matrix $D_0$, i.e.,
\eqref{eq:wilson} with $m_{q} = 2(d-1)$, for $n=32$ with $\beta = 3$ and $\beta =6$ are provided in 
Figure~\ref{fig:specD0b3}, and 
Figure~\ref{fig:eigsD0} contains plots of the 16 smallest
eigenvalues of $D_0$ for $N=128$ and $\beta = 6$ for nine distinct gauge
field configurations.  We note that in all cases 
the eigenvalues of $D_0$ have a positive real part, which holds for all of the problems considered
in this paper.  Additionally, as
the plots in Figure~\ref{fig:specD0b3} illustrate, when the value of $\beta$ decreases the
eigenvalue with minimal real part moves away from the origin and the
eigenvalue with maximal real part moves closer to the origin.

The $\gamma_5$-symmetry of the Wilson matrix implies
that the eigenvalues of $D$ are either real or appear in complex conjugate pairs.
Specifically,  define $\gamma_5 =  \operatorname{diag}(1,-1)$ and set 
$\Gamma_5 = I_{N^2} \otimes \gamma_5$.  Then, $\Gamma_5^H\Gamma_5 =
\Gamma_5^2=  I$ and $$\Gamma_5 D = D^H \Gamma_5 \quad \mbox{or} \quad
D = \Gamma_5 D^H \Gamma_5.$$  Now, 
if $v_\lambda$ denotes a right eigenvector of $D$ to the eigenvalue $\lambda \neq \bar{\lambda}$, we see that
$\Gamma_5 v_\lambda$ is a left eigenvector to the eigenvalue
$\bar{\lambda}$, i.e.,
\begin{equation}\label{eq:rightevleftev}
  D v_\lambda = \lambda v_\lambda \iff ( \Gamma_5 v_\lambda)^H D  = \bar{\lambda} (\Gamma_5 v_\lambda)^H,
\end{equation} Thus, to each
right eigenpair $(\lambda,v_\lambda)$ there corresponds a left eigenpair 
$(\bar{\lambda},\Gamma_5 v_\lambda)$, and the spectrum of $D$ is
symmetric with respect to the real axis. 
More generally, since $\Gamma_{5}$ is unitary, we have $\norm{\Gamma_{5}x} = \norm{x}$
for any $x$, and the $\gamma_{5}$-symmetry yields, in addition,
that $\norm{Dx} = \norm{D^{H}\Gamma_{5} x}$. Thus for a general near
kernel component, $x\in \linspace[n]{C}$, we find that 
$$\frac{\|Dx\|}{\|x\|} \approx 0 \iff \frac{\|D^H\Gamma_5
  x\|}{\|\Gamma_{5} x\|} \approx 0.$$   

Overall, as the plots in Figures~\ref{fig:specD0b3} and~\ref{fig:eigsD0} illustrate, 
depending on the choice of the shift $m$, the resulting system matrix can have a large number of
eigenvalues that are close to zero and, thus, potentially a large
number of small singular values.  This observation motivates the use of the 
multigrid eigensolver as an approach to efficiently compute
several near kernel components simultaneaously 
in the proposed bootstrap AMG setup algorithm.
\begin{figure}
  \begin{minipage}[t]{.5\textwidth}
    \begin{center}
    \subfigure[Spectrum of $D_0$ for $\beta
    =3$]{\includegraphics[width=.75\textwidth]{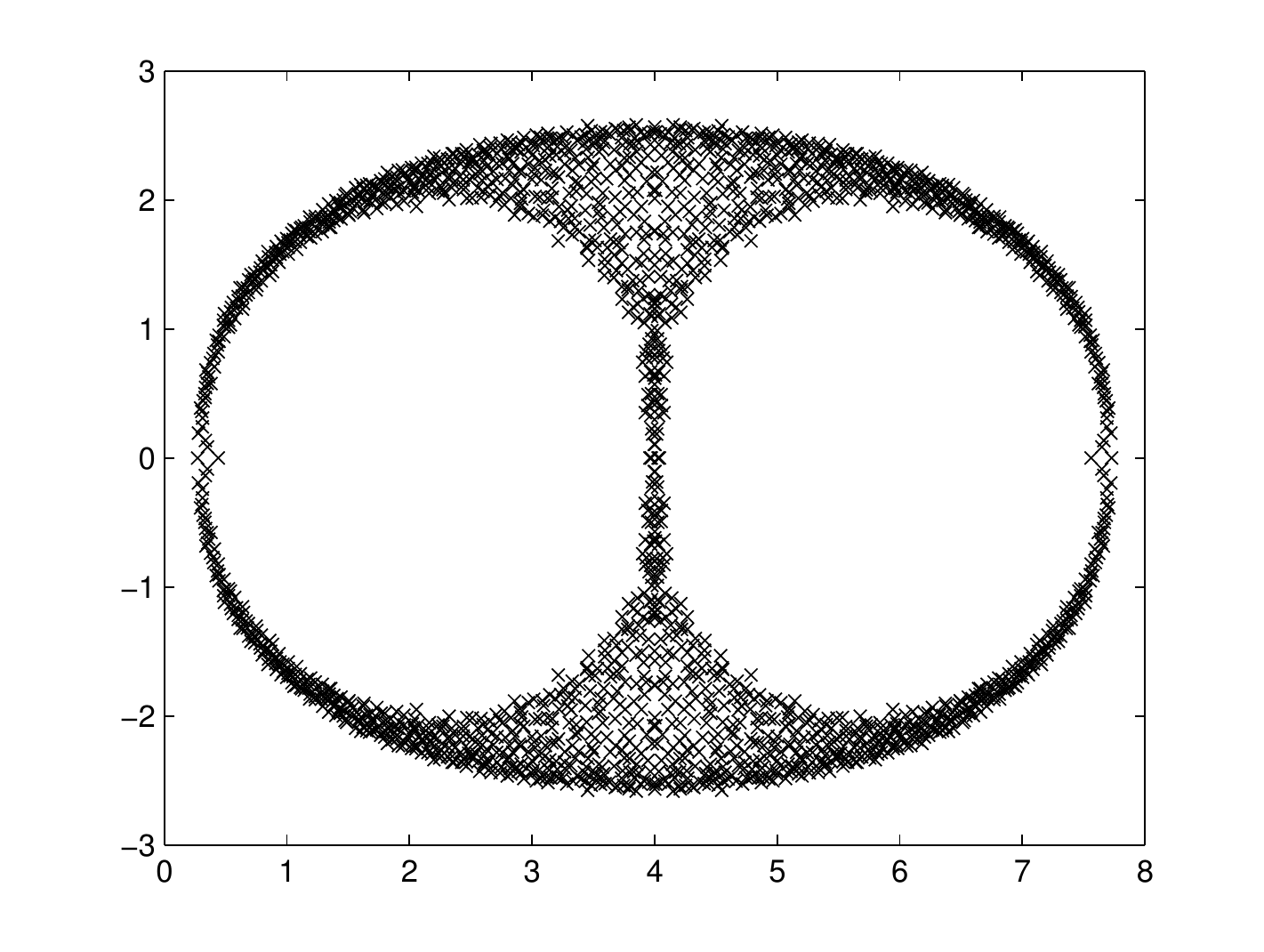}}
  \end{center}
  \end{minipage}\hfill
  \begin{minipage}[t]{.5\textwidth}
    \begin{center}
    \subfigure[Spectrum of $D_0$ for
    $\beta=6$]{\includegraphics[width=.75\textwidth]{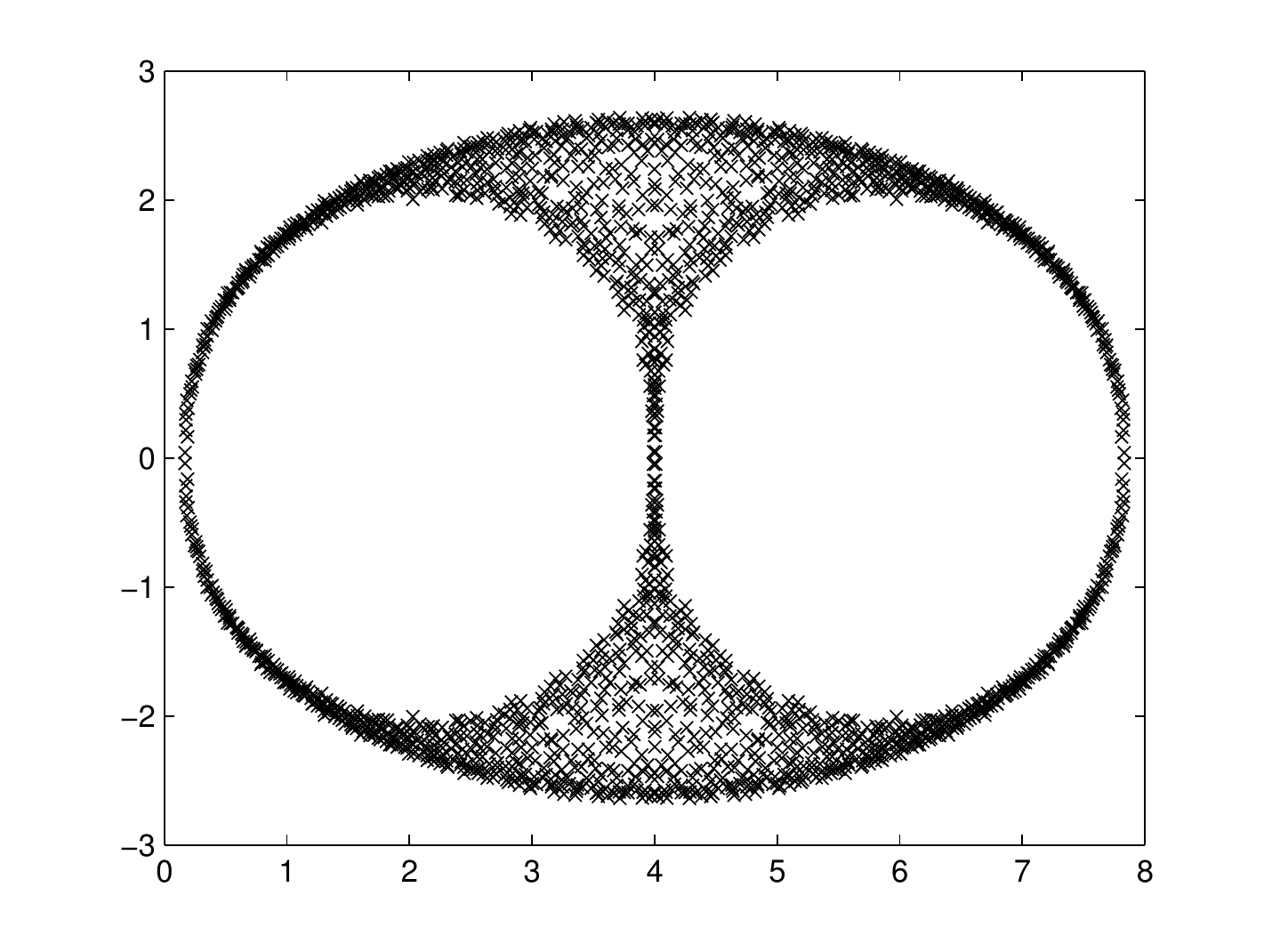}}
    \end{center}
    \end{minipage}
    \caption{Spectra of the mass-less Wilson matrix $D_0$ for $N=32$ and
      $\beta = 3, 6$, configuration 13,000.\label{fig:specD0b3}}
\end{figure}

\begin{figure}
\centerline{
  \includegraphics[width=.85\textwidth]{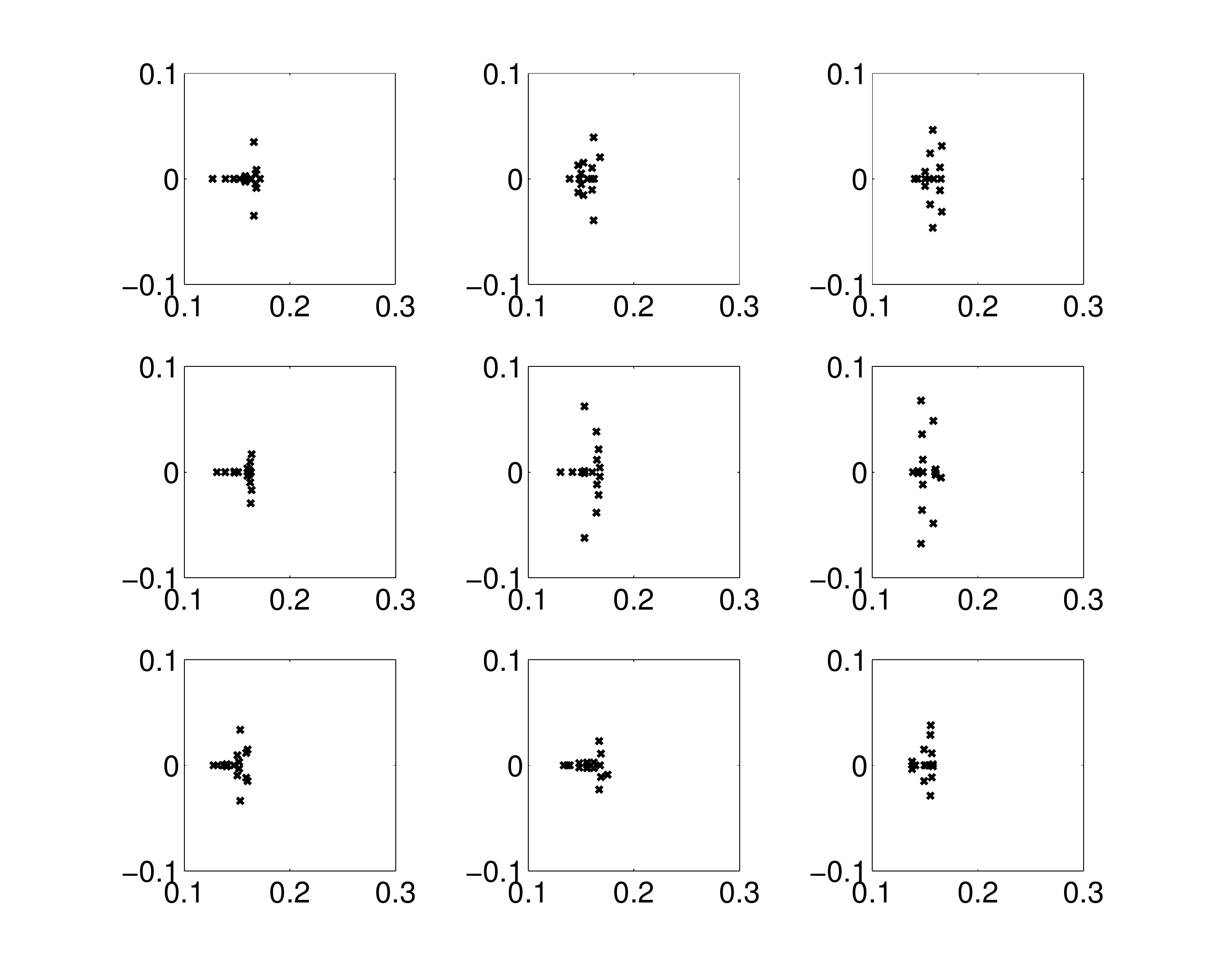}
}
\caption{Smallest sixteen eigenvalues of the Wilson matrix $D_0$ for $N=128$ and $\beta
  = 6$. The plots from left to
right and then top to bottom correspond to the nine gauge field configurations 
11,000, 12,000, ..., 19,000.}\label{fig:eigsD0}
\end{figure}

\subsection{Failure of Krylov methods for the Wilson matrix}\label{Wilson:krylov}

Typically, a standard Krylov method (e.g., BiCG, GMRES, CGNR) is used 
to solve the linear systems~\eqref{eq:dirac} arising throughout a lattice QCD simulation.
The large condition number of the Wilson matrices that result from
physically relevant choices
of the shift $m$ lead to slow convergence of these methods, as shown in the plots
on the left in Figure~\ref{CGNor}, where we report results of CGNR and
restarted GMRES  with a restart value of $32$ (GMRES($32$))
applied to a series of linear systems involving 2-dimensional Wilson matrices.  
For both methods, we see that the solver requires a large number of
iterations to drive the residual down to the given tolerance and that
GMRES($32$) reaches the maximum number of iterations before reaching
the convergence criteria in many cases.  Here, the maximum number of
iterations is limited to 4096 and the solver stops if it reaches the
prescribed tolerance of $10^{-8}$ reduction in the relative residual
or this number of iterations.

Moreover, as illustrated in the plots on the right in
Figure~\ref{CGNor}, even when the algorithm stops successfully, the actual error is large compared to the
residual.  In fact, we observe that the relative $\ell_2$ norm of the
error is up to six orders of magnitude larger than the relative $\ell_2$ norm of the final residual.  
Of course, decreasing the tolerance for the norm of the residual for either method should further reduce the error, but would result in an even larger number of iterations.
Overall, these results are representative of the performance of
standard Krylov methods applied to the Wilson matrix.  
We mention in addition that although $m$ is set so that
$\eta_{\min} (D) > 0$, $\lambda_{\min} (A)$, the smallest
eigenvalue of the gauge Laplacian block from~\eqref{eq:bst}, can become 
negative which complicates the use of block preconditioners (e.g.,
Uzawa type schemes) for the solution of Wilson matrices for physically
relevant choices of $m$.
\begin{figure}
  \begin{minipage}{.5\textwidth}
    \begin{center}
      \subfigure[CGNR results for Wilson matrices.]{\includegraphics[width=\textwidth]{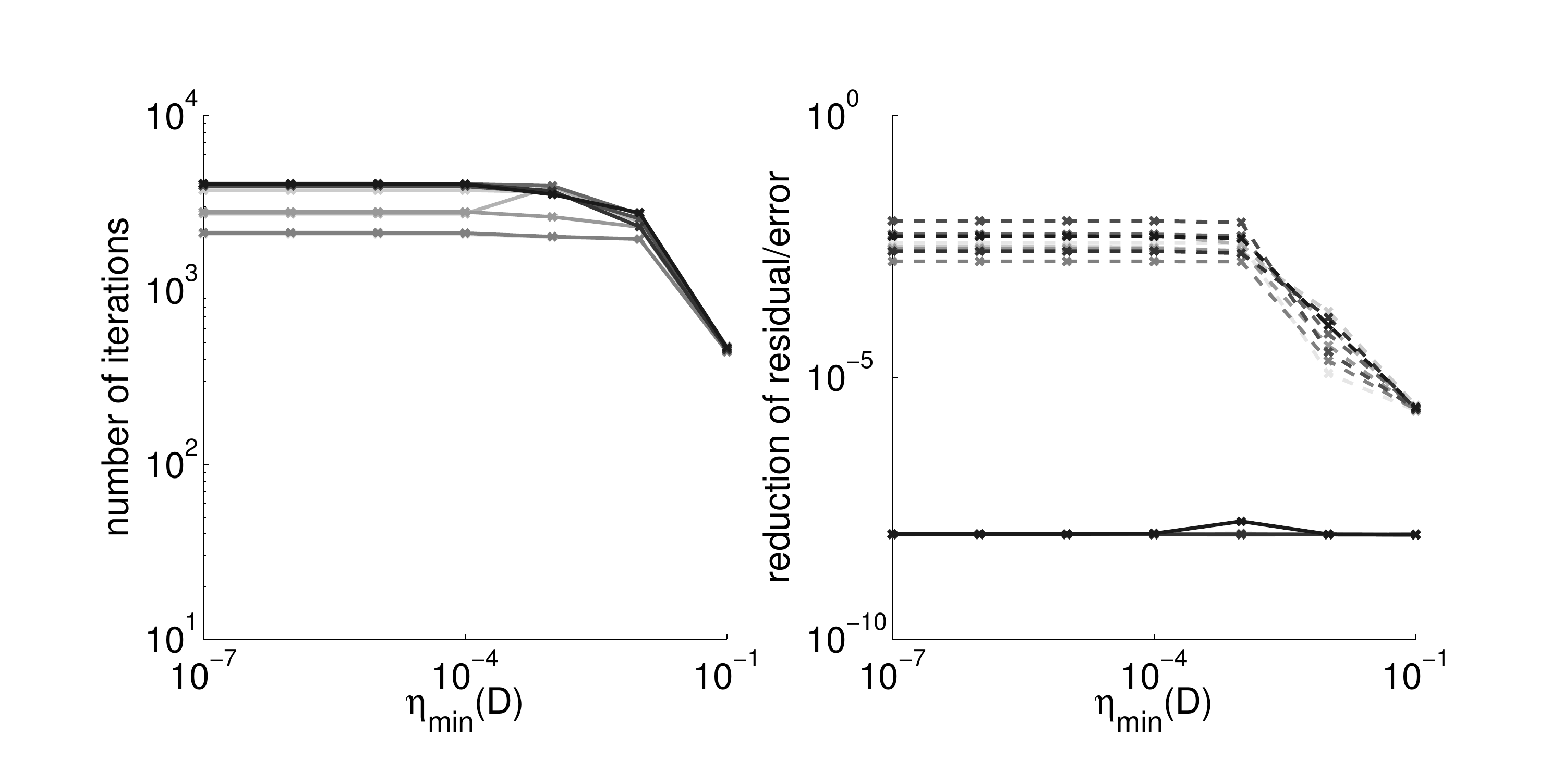}}
    \end{center}    
  \end{minipage}\hfill
  \begin{minipage}{.5\textwidth}
    \begin{center}
            \subfigure[GMRES($32$) results for Wilson matrices.]{\includegraphics[width=\textwidth]{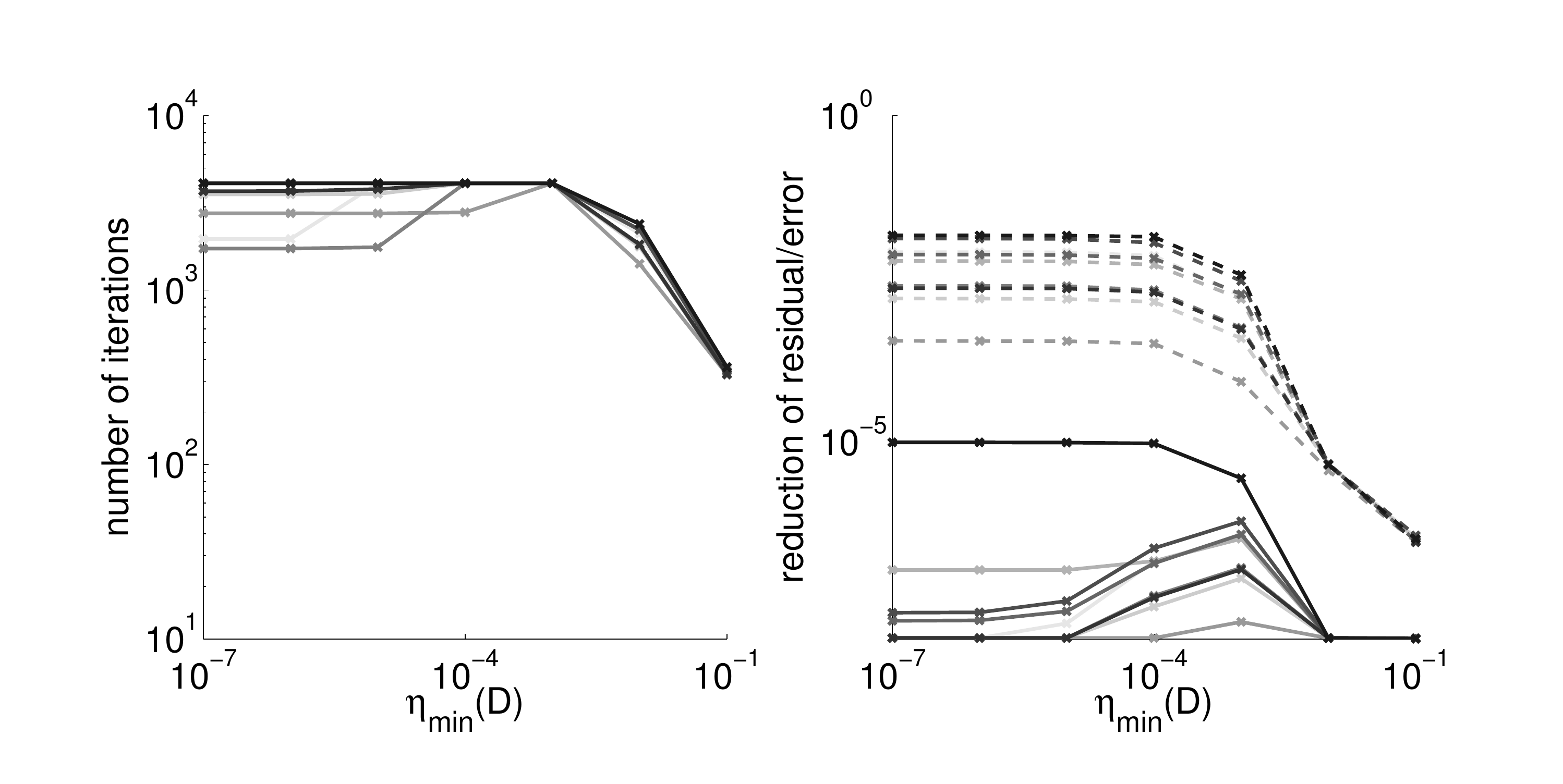}}
         \end{center}    
  \end{minipage}  
  \caption{Results of CGNR and GMRES($32$) applied to Wilson matrices
    for $N=128$ and $\beta = 6$. The results for the 
    light to dark lines correspond to
    different gauge field configurations, going from light to dark with increasing configuration number. 
    On the left of each subplot, the number of iterations needed to reduce the $\ell_2$ norm of the relative residual by a factor 
    $10^{-8}$ is plotted against different values of $\eta_{\min} (D)$
    defined in~\eqref{eq:etamin}, corresponding to different choices of the
    diagonal shift $m$. On the right of each subplot, the resulting
    relative residuals (solid lines) and relative errors (dashed lines) are plotted against $\eta_{\min} (D)$.}\label{CGNor}
\end{figure}

\section{Bootstrap geometric-AMG for the Wilson matrix}\label{BAMG}
In this section, we develop a bootstrap approach for linear systems
with the non-Her\-mi\-tian Wilson matrix and study it in detail.  The
algorithm we consider here combines a bootstrap setup to compute the
test vectors used in defining interpolation with an adaptive step that
applies the existing solver to an appropriate initial guess to update
the test vector(s).  We note that all arguments made in this section 
also carry over to other discretizations of the Dirac equation as long as they satisfy
the $\gamma_5$-symmetry. 

\subsection{Kaczmarz relaxation} 
As the smoother in the proposed BAMG algorithm we consider Kaczmarz relaxation.  
For the linear system, $D \psi =b$, with the non-Hermitian Wilson
matrix $D$ the  
Kaczmarz iteration is based on the equivalent formulation involving the normal equations
\begin{equation}\label{eq:ne}D^H D \psi =D^Hb.
\end{equation}
Given an approximation to the solution, $\psi$, of \eqref{eq:ne}, one iteration of the basic Kacmarz iteration for this formulation reads: 
\begin{equation*}\label{eq:nrgs}
\psi \leftarrow \psi + s_i e_i, \quad i = 1, ..., n,
\end{equation*}
where $e_i$ is the $i$-th Euclidian basis vector and $s_i$ is chosen so that the corresponding component of the 
residual vanishes:
$$ \langle D^H b - D^H D(\psi + s_i e_i), e_i \rangle = 0.$$  Now, setting $ r = b - D \psi$ gives
$$ \langle D^H (r + s_i De_i), e_i \rangle = 0, \quad \mbox{implying} \quad
 \displaystyle{s_i = \frac{\langle r,D e_i \rangle}{\|D e_i\|_2^2}}.$$
Thus in practice Kaczmarz relaxation can be realized using column
access to the entries of $D$ only 
and the arithmetic complexity of a single iteration depends only on the number of non-zero entries in $D$. 

Local mode analysis for Kaczmarz relaxation suggests that it is widely
applicable as a smoother, although it is often less efficient with
respect to the actual smoothing rate (see~\cite[Section
4.7]{Trottenberg:2000:MUL:374106}).  Further analysis of the smoothing
properties of the Kaczmarz iteration applied to general rectangular
systems is found in~\cite{IOPORT.05602071}.  While local mode analysis
is not applicable to the Wilson matrix due to the local variations in
the gauge fields and, hence, its off-diagonal entries, 
the numerical results we provide for the proposed multigrid solver
with Kaczmarz smoother suggest its robustness for the Wilson discretization of the Dirac
equation.  
We address its suitability in the bootstrap AMG setup process for computing singular value triplets in the subsequent section.

We note that the Kaczmarz iteration as defined above updates the
unknowns in sequential order 
$i = 1,\ldots,N^d$, referred to as a {\em forward sweep}.  Thus, 
as is, the method is not amenable to parallel computing.  
Formulating a parallel version using an
appropriate {\em coloring} strategy to order the updates 
of the unknowns is, however, straightforward since Wilson's
discretization of the Dirac equation is formulated on a regular grid
with nearest-neighbor coupling. 

\subsection{Geometric coarsening}\label{sec:coarsening}
\begin{figure}
  \begin{minipage}[t]{.225\textwidth}
    \subfigure[Odd-even coarsening of the grid]{\includegraphics[width=\textwidth]{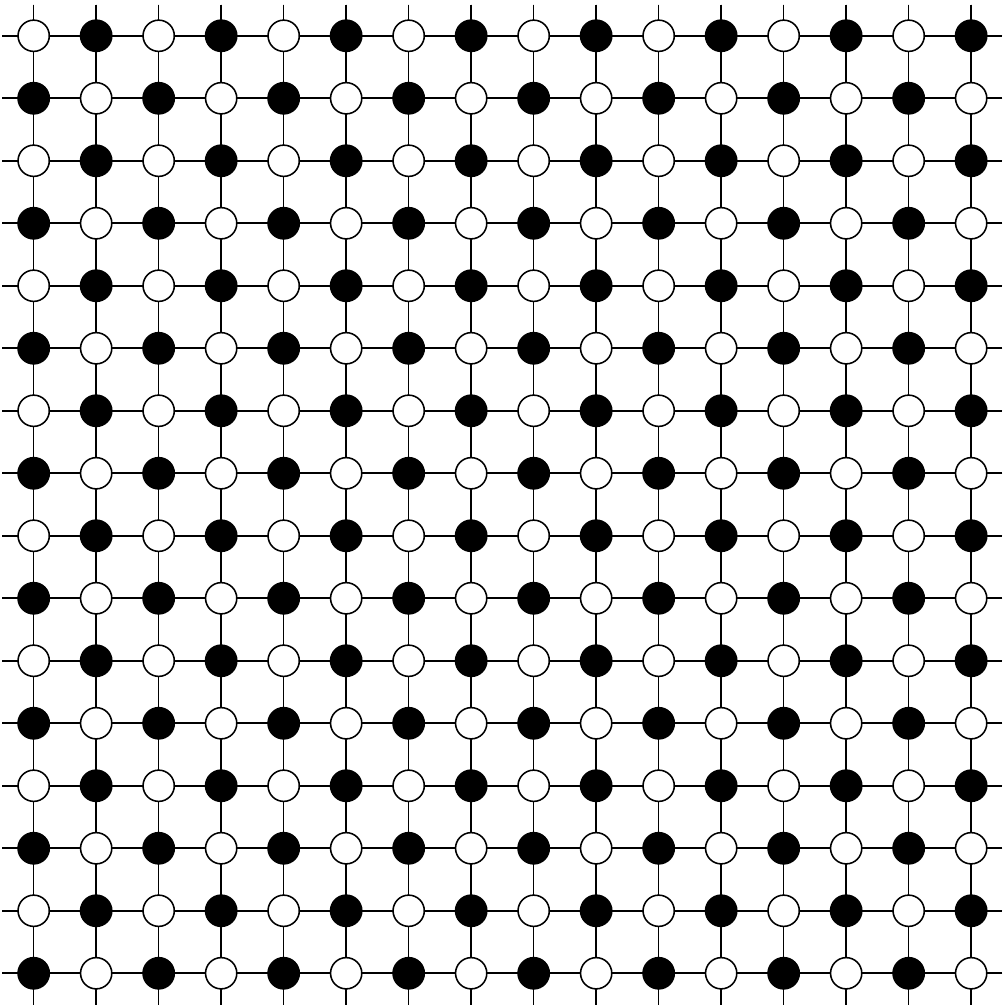}}
  \end{minipage}\hfill
  \begin{minipage}[t]{.225\textwidth}
  \subfigure[Structure of the odd-even reduced operator]{\includegraphics[width=\textwidth]{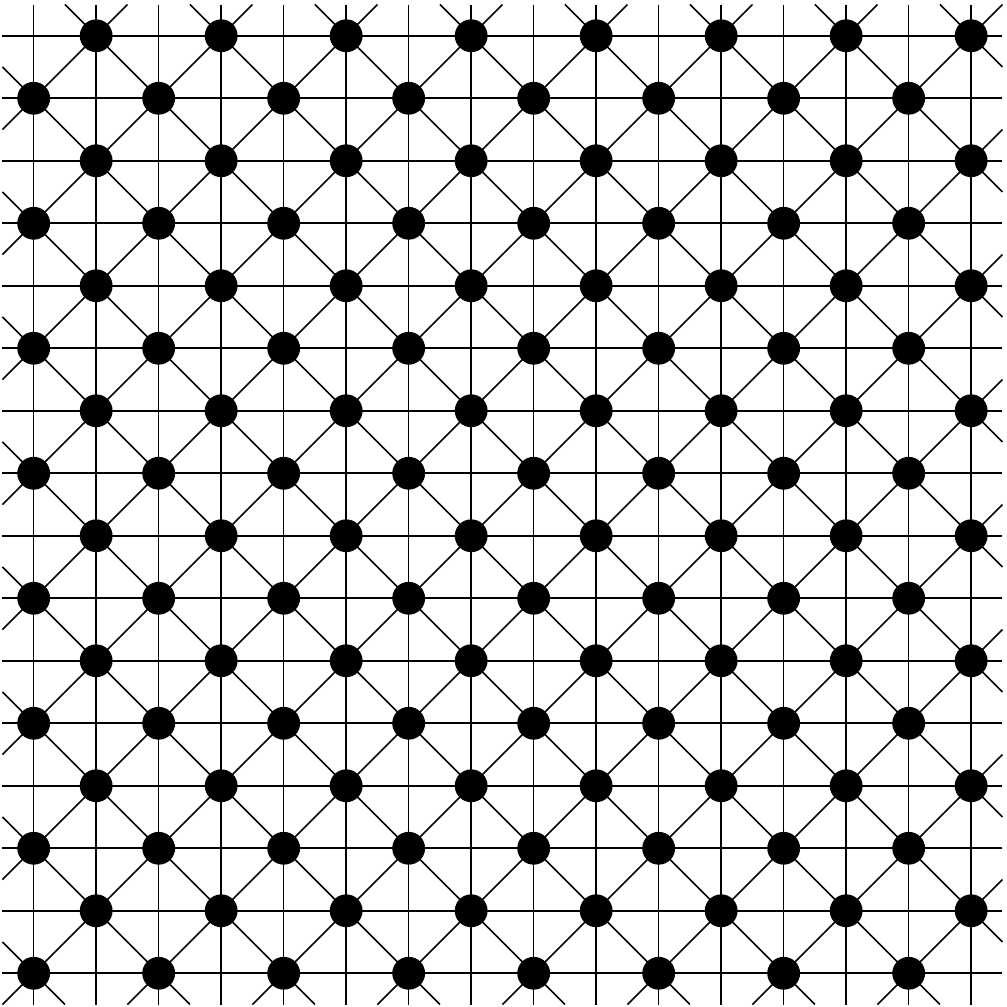}}
  \end{minipage}\hfill
  \begin{minipage}[t]{.225\textwidth}
   \subfigure[Full coarsening of the ``even'' grid]{\includegraphics[width=\textwidth]{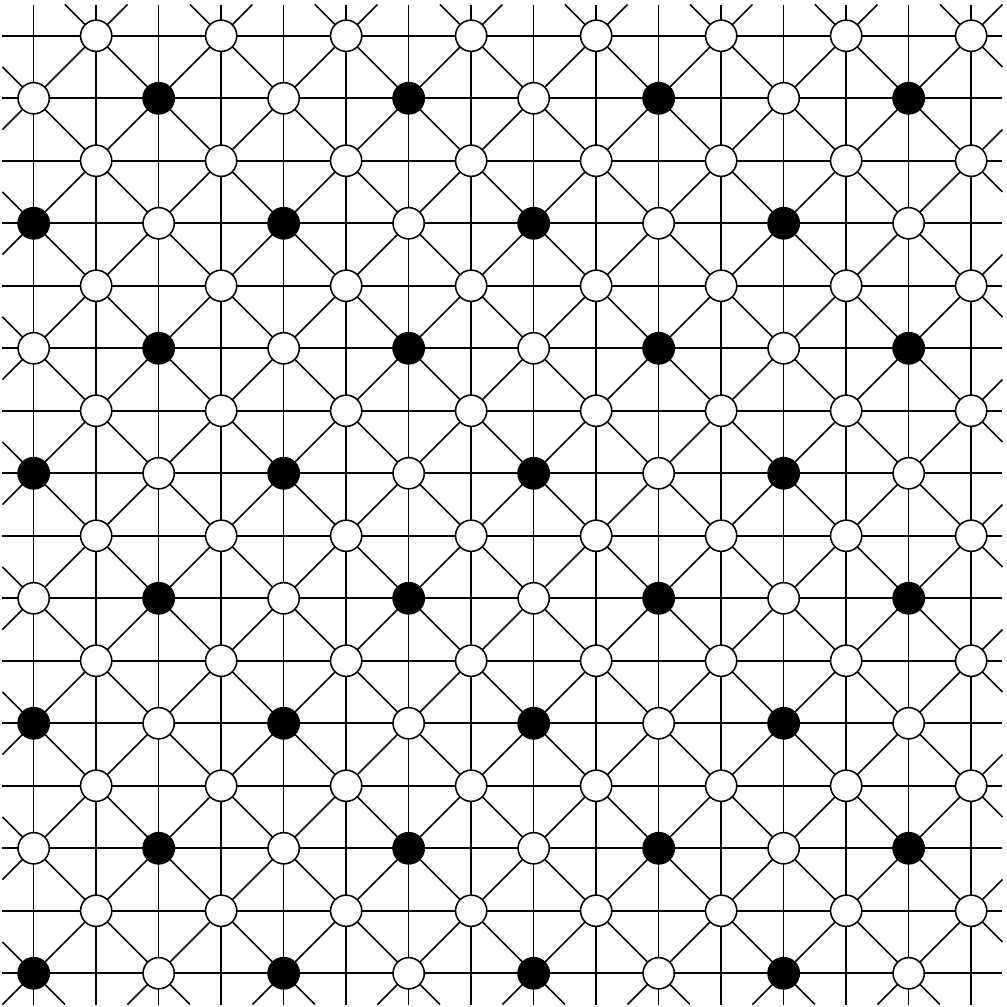}}
  \end{minipage}\hfill
  \begin{minipage}[t]{.225\textwidth}
  \subfigure[Structure of the next coarser grid assuming next-neighbor
  interpolation]{\includegraphics[width=\textwidth]{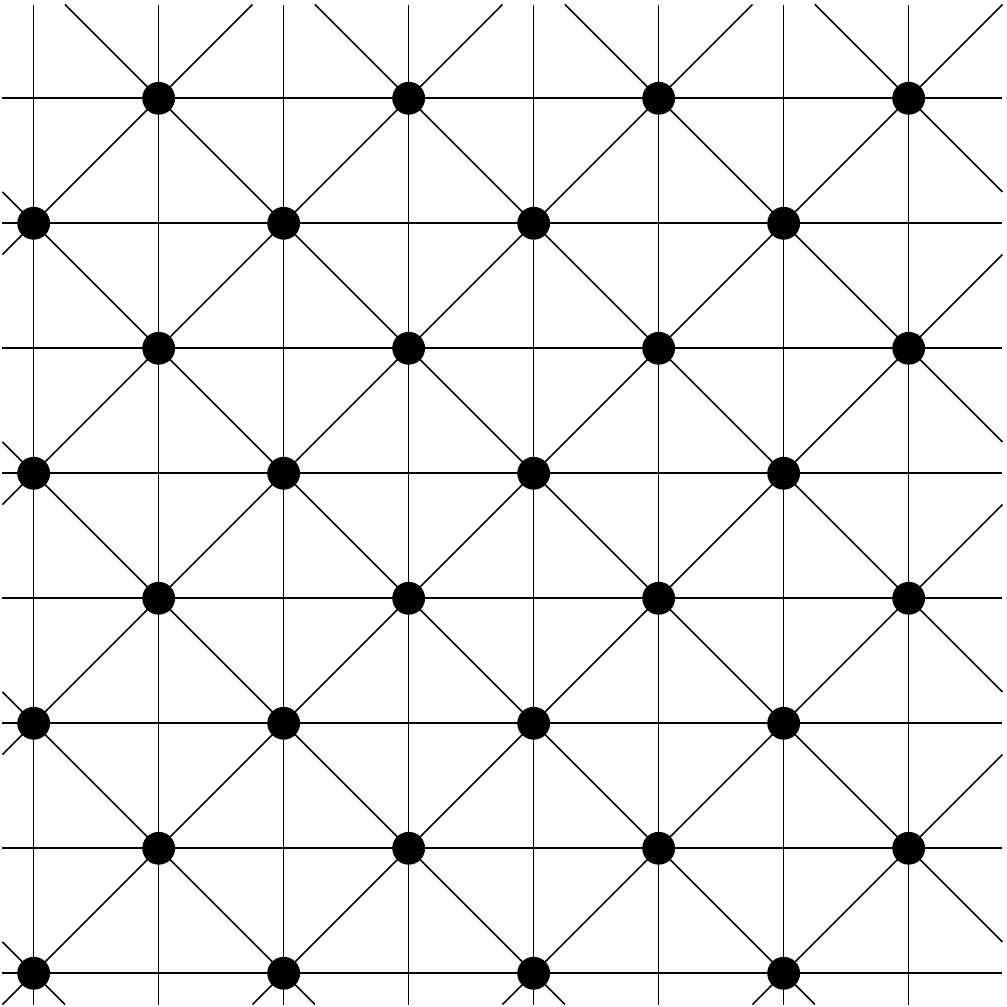}}
  \end{minipage}
  \caption{The odd-even reduction and full coarsening of the grid of even points. The circles denote grid points
    and the edges connections among them, defined according to the
    graph of $D$ and the Schur complement $\widehat{D}$
    resulting from odd-even reduction. In~(a) and~(c) white points
    correspond to coarse points.   \label{fig:oddevencoarsening}}
\end{figure} 
The nearest-neighbor finite difference scheme and regular cubic grid
used in the Wilson discretization of the Dirac equation
allows for an odd-even reduction (or coarsening), which is typically applied when developing solvers for this 
system~\cite{Osborne_POS_2009,MILC}. 
Let a grid point $(x,y)\in \{1,\ldots,N\}^2$ be labeled as \textit{even} if $x+y$ is even and
as \textit{odd} otherwise (see~Fig.~\ref{fig:oddevencoarsening}). In case the total number of grid points, $n$, is
even, then the number of odd and even points is exactly $n/2$. Any
vector $\psi \in \linspace[n]{C}$ can be written as
$ \psi =  \left(\begin{matrix} \psi_{e}^T, \psi_{o}^T\end{matrix}\right)^T$ by numbering the even points before the odd
ones. Using the same numbering scheme for the rows and columns of the Wilson matrix gives
\begin{equation*}
  D = \left(
    \begin{matrix}
      D_{ee} & D_{eo} \\ D_{oe} & D_{oo}      
    \end{matrix}
  \right).
\end{equation*}
Now, since $D$ couples only nearest-neighbors on the grid
(see~Fig.~\ref{fig:oddevencoarsening}~(a)), the blocks
$D_{ee}$ and $D_{oo}$ are diagonal. Specifically, 
$D_{ee} = D_{oo} = c\cdot I$ for $c\in \linspace{R}$ and upon scaling by $c^{-1}$
the constant can be set as $c=1$. Define the operator 
\begin{equation}\label{eq:wilsonSchur}
\widehat{D} = I -
D_{eo}D_{oe},
\end{equation}
i.e., $\widehat{D}$ is the 
the Schur complement of $D$ with respect to 
the even points,
referred to as the odd-even reduced matrix. 
With it the solution of the linear system
$
  D\psi = b
$ can be calculated in the following two steps. \medskip
\begin{enumerate}
\item Solve $\widehat{D}\psi_{e} = b_{e} - D_{eo}b_{o}$.
\item Compute $\psi_{o} = \psi_{o} - D_{oe}x_{e}$.
\end{enumerate}\medskip

From the form of the Schur complement
in~\eqref{eq:wilsonSchur}, we see that a matrix vector 
multiplication with $\widehat{D}$ requires
roughly the same number of floating point operations as the
multiplication with $D$. 
However, solving systems with $\widehat{D}$, instead of the original
matrix $D$, typically reduces the total number of CGNR iterations
by a factor of two. This is illustrated in Figure~\ref{CGNodd},
where results of CGNR applied to the system with $\widehat{D}$ for
$N=128$ and $\beta=6$ are reported.  Here, we use the same nine gauge
configurations considered in the tests of CGNR applied to the
unreduced Wilson matrix $D$ reported in Figure~\ref{CGNor} and we 
see that CGNR applied to $\widehat{D}$ needs half as many iterations on average in order to reduce the 
$\ell_2$ norm of the relative residual by the factor $10^{-8}$.

\begin{figure}
\begin{center}
\includegraphics[width=.6\textwidth]{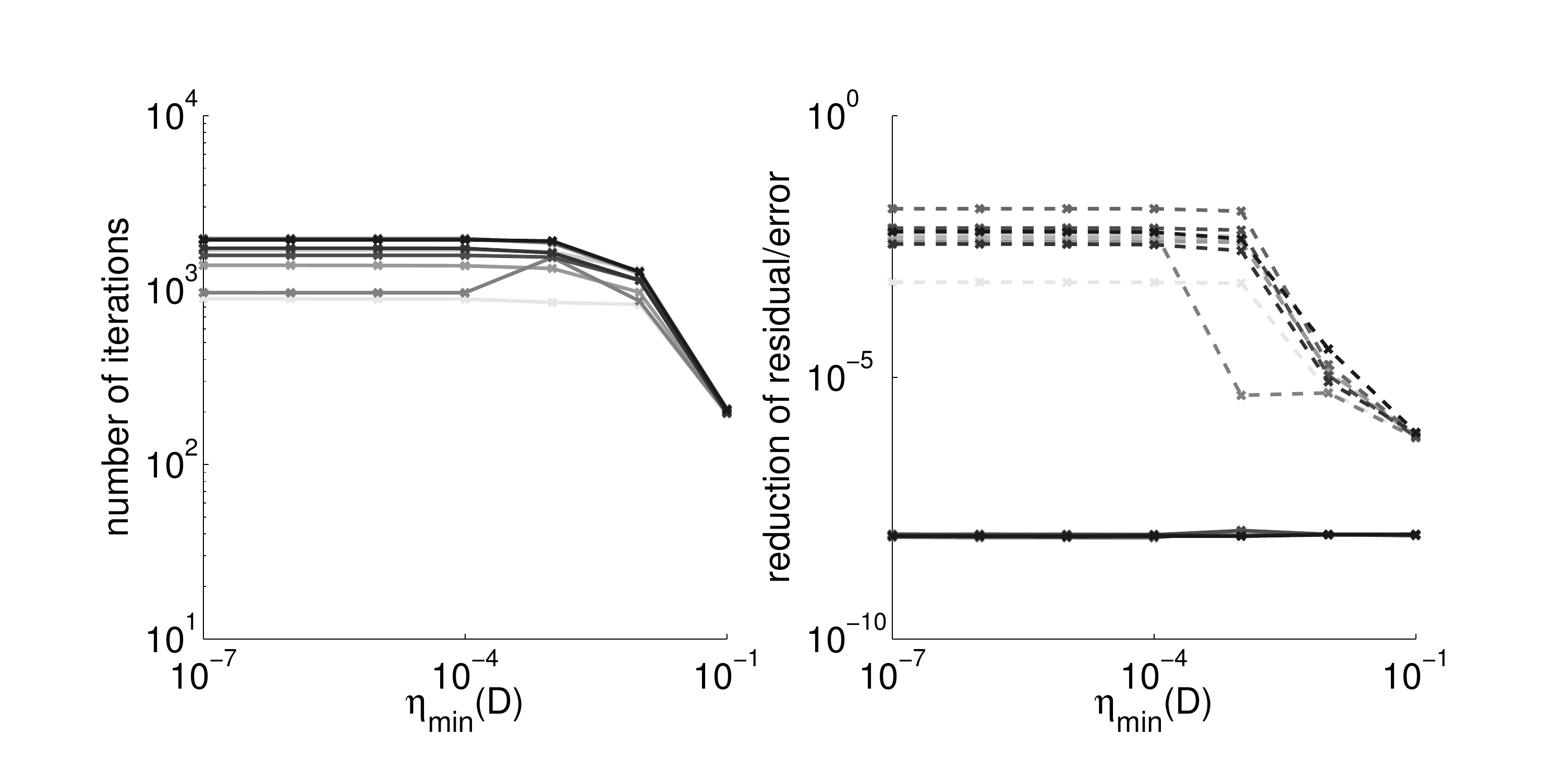}
 \caption{Results of CGNR applied to the odd-even reduced system with 
$N=128$ and $\beta = 6$. 
The results for the light to dark lines correspond to
different gauge field configurations, going from light to dark with increasing configuration number. 
On the left, the number of iterations needed to reduce the relative residual by a factor 
$10^{-8}$ is plotted against 
$\eta_{\min}(D)$ defined in~\eqref{eq:etamin}, corresponding to different
diagonal shifts $m$. On the right, the resulting
relative residuals (solid lines) and relative errors (dashed lines) are plotted
against $\eta_{\min}(D)$.\label{CGNodd}}
\end{center}
\end{figure}

Since CGNR for the system with $\widehat{D}$ is the default solver in
various lattice QCD simulation codes~\cite{Osborne_POS_2009,MILC}, our
construction of the proposed BAMG algorithm is based on the odd-even
reduced matrix $\widehat{D}$.  Fig.~\ref{fig:oddevencoarsening}~(c) illustrates the full coarsening strategy we use for $\widehat{D}$ 
on the first even (coarse) grid and on all subsequent grids of the AMG
hierarchy, namely, we define every other grid point in every
dimension as a coarse grid point.  Now, if in addition nearest-neighbor (restriction and) interpolation 
and a (Petrov) Galerkin coarse-grid construction are used, then it follows that the resulting coarse-grid operator 
again has the same sparsity structure, as depicted in Fig.~\ref{fig:oddevencoarsening} (d). 

An important observation for the derivation of the bootstrap setup cycle presented in the next section is that the Schur complement on the even grid also satisfies the $\gamma_5$-symmetry, 
so that the results we present for $D$ hold for $\widehat{D}$ as well. 
To see that $\widehat{D}$ satisfies the $\gamma_5$-symmetry, 
consider the corresponding block form of $\Gamma_5$, i.e., 
\begin{equation*}
  \Gamma_{5} = \begin{pmatrix}
    \Gamma_o   &    \\
    &  \Gamma_e
  \end{pmatrix},
\end{equation*} where $\Gamma_{o}, \Gamma_{e}$ have the same structure
as $\Gamma_{5}$. Direct computation shows that we have $D_{oe}^{H}\Gamma_{o} = \Gamma_{e}D_{eo},
D_{eo}^{H}\Gamma_{e} = D_{oe}\Gamma_{o}$ and, thus,
\begin{equation}\label{eqevZ2}
\widehat{D}^H \Gamma_e = \Gamma_e \widehat{D}.
\end{equation}
\subsection{Least squares restriction and interpolation}\label{sec:lsp}
The first main component of the bootstrap AMG algorithm is its use of a least
squares process to form the restriction and interpolation
operators. The least squares restriction $R$ and interpolation $P$ are defined to fit
collectively sets of left and right test vectors, respectively. The
test vectors used in these fits characterize the near kernel of the system matrix. In order to
simplify notation, we present the least
squares process first in a more general setting, i.e., for general sets 
of interpolatory variables $\mathcal{C}_i$, and 
then discuss the specific definition we use in the proposed algorithm,
which we base on the block-spin structure of the Wilson matrix given in~\eqref{eq:bst}.
Further, we present only the construction of $P$ and note that
$R^H$ is obtained in complete analogy. 

Let $\widehat{\Omega}$ denote the set of all variables of the linear
system~\eqref{eq:dirac}, then given a set of coarse variables $\coarsevar
\subset \widehat{\Omega}$, e.g., defined by full coarsening, we set $\finevar = \widehat{\Omega}
\setminus \coarsevar$. Further, let the set of interpolatory variables for 
a fine variable $i\in \mathcal{F}$ be denoted by 
by $\mathcal{C}_i$, typically consisting of neighboring coarse
variables. Then the structure of $P$ is defined by 
\begin{equation*}\label{eq:nearneigh} 
P_{ij} = 0 \: \mbox{for} \: j \notin \mathcal{C}_i.
\end{equation*}

Once the sets of interpolatory variables, $\mathcal{C}_i$, and a set of test vectors, $\tvV = \{\widetilde{v}^{(1)}, \ldots,
\widetilde{v}^{(k)}\} \subset \mathbb{C}^n$, have been determined, the $i$th
row of $P$ for $i\in \finevar$, denoted by $p_i$, 
is then defined as the minimizer of the local least squares problem: 
\begin{equation}\label{eq:LSfuncrowi}
\mathcal{L}(p_i) = 
\sum_{\kappa=1}^k\omega_k\left(\widetilde{v}_{\{i\}}^{(\kappa)} - \sum_{j\in \coarsevar_{i}} \left(p_{i}\right)_{j} \widetilde{v}_{\{j\}}^{(\kappa)}\right)^{2} \mapsto \min.
\end{equation} 
Here, the notation $\widetilde{v}_{\widetilde{\Omega}}$
denotes the canonical restriction of the vector $\widetilde{v}$ to
the set $\widetilde{\Omega} \subset \widehat{\Omega}$, e.g.,
$\widetilde{v}_{\{i\}}$ is simply the $i$th entry of $\widetilde{v}$.
The weights $\omega_{\kappa}>0$ are chosen to reflect the importance
of the corresponding test vector 
(e.g., its $A$-norm $\|\widetilde{v}\|^2_A = \langle
A\widetilde{v},\widetilde{v}\rangle$ when $A$ is Hermitian and
positive definite).  We give our specific choice of the weights in the next section.
Conditions on the uniqueness of the solution to minimization problem
\eqref{eq:LSfuncrowi} and an explicit form of the minimizer have been
derived in \cite{BAMG2010}.

In the case where there is more than one variable per grid point\footnote{This is the case for the 2-dimensional Wilson matrix where two variables, corresponding to
the two spin components, are defined at each grid point.}, the interpolatory sets used in the
least squares process are defined as follows. The sets $\mathcal{F}$ and $\mathcal{C}$ are defined
with respect to grid points rather than variables. That is, all
variables at a given grid point are marked either fine or coarse
collectively. Then, given a grid point $i\in \mathcal{F}$ with associated
variables $i_{1},\ldots, i_{m}$, interpolation is built
independently for each variable. The set of interpolatory 
variables $C_{i_{*}}$ for each variable is then composed of a subset or
all of the variables defined at grid points $j\in \mathcal{C}$ in the
neighborhood of grid point $i$. 
In the following sections, we show that by imposing further conditions
on the sets of interpolatory variables $C_{i_{*}}$ it is  
possible to preserve the spin structure and $\gamma_5$-symmetry on
coarse grids, which is an important feature in our proposed bootstrap setup algorithm.

\subsection{The bootstrap multigrid setup cycle}\label{sec:bootamg}  
{The second main component of the bootstrap AMG setup is the bootstrap cycle used for computing the test
vectors that are needed in the least squares interpolation process \eqref{eq:LSfuncrowi}. 
The proposed bootstrap cycle setup uses two complementary processes to compute the 
sets of tests vectors: \medskip
\begin{enumerate}[$1.$]
\item a solver applied to appropriately formulated homogenous problems on different grids, 
as in adaptive AMG,
\item a multigrid eigensolver.
\end{enumerate}\medskip
As discussed in detail below, the main idea in the latter multigrid
eigensolver is to use appropriate mass matrices to formulate
generalized eigenproblems on coarser grids in such a way that these
generalized  eigenproblems can be directly related to the finest-grid
eigenproblem.} 

Following the reasoning in Section \ref{sec:intro}, we use the fact that these smooth error components
are characterized by left and right singular vectors with small
singular values.  That is, in analogy to the Hermitian and positive definite case, we seek to construct test vectors that capture the
algebraically smooth components of the error, i.e., vectors $x$ such
that $Dx \approx 0$ or $D^H x \approx 0$.  
Let $D = U \Sigma V^H$ be the SVD of the non-Hermitian Wilson matrix
$D$, where $U = [u_{1}\mid\ldots\mid u_{n}] \in \mathbb{C}^{n \times n}$ and $V=[v_{1}\mid\ldots\mid v_{n}] \in \mathbb{C}^{n \times n}$ are unitary matrices.
Then, the left and right singular vectors and corresponding singular values 
are given by the triplets $(\sigma_{i}, u_{i},v_{i})$ that satisfy the equations 
\begin{eqnarray*} 
D v_{i} &=&\sigma_{i} u_{i} , \\
D^H u_{i} &=&\sigma_{i} v_{i}. 
\end{eqnarray*}
Now, since these equations are equivalent to the Hermitian and indefinite
eigenvalue problem (cf.~\cite{golub_kahan_1965,lanczos_1961})
\begin{equation}\label{eq:symev}
  \begin{pmatrix}
    & D    \\
    D^H & 
  \end{pmatrix}
  \begin{pmatrix}
    U & U  \\    
    V   & - V
  \end{pmatrix}  = 
  \begin{pmatrix}
    U & U  \\    
    V  & - V
  \end{pmatrix}
  \begin{pmatrix}
    \Sigma & \\
    & -\Sigma
  \end{pmatrix}
  ,
\end{equation} 
it follows that a Petrov Galerkin bootstrap AMG process for the non-Hermitian matrix $D$
can be reformulated as a Galerkin bootstrap AMG process for the Hermitian
system~\eqref{eq:symev}, to which the algorithm developed
in~\cite{BAMG2010} can be applied. This approach was proposed for the
computation of singular triplets in~\cite{hans_arxiv} and is described
in the following.

Starting with the finest-grid system $D_{1}$, the Petrov Galerkin bootstrap setup cycle 
begins by applying relaxation (e.g., Kaczmarz iterations) to the homogeneous systems
\begin{equation}\label{homog}
D_l  \: \widetilde{v}_l^{(\kappa)} =0, \quad \text{and} \quad D_l^H \: \widetilde{u}_l^{(\kappa)} =0, \quad \kappa
= 1,..., k_r, \quad l=1,...,L-1 ,
\end{equation}
to compute some initial sets of right $\tvV^r_l = \{\widetilde{v}^{(\kappa)}_{l}, \: \kappa = 1,\ldots,k_r\}$ and left  $\tvU^r_l = \{\widetilde{u}^{(\kappa)}_{l}, \: \kappa = 1,\ldots,k_r\}$ test vectors used in constructing the least
squares interpolation $P_{l+1}^{l}$ and restriction $R^{l+1}_{l}$
operators, $l=1,...,L-1 $
, respectively, 
and thereby also the corresponding coarse-grid operators.  
On the finest grid, $k_r$ distinct random vectors are used as the
initial guesses for the Kaczmarz iterations applied to each of the two
systems in \eqref{homog}.
On all subsequent grids except the coarsest, $ l=1,...,L-1 $, 
the resulting relaxed vectors computed on finer grids $l$
are restricted to the coarser grids $l+1$ and used as the initial guesses
for Kaczmarz applied to the two systems in \eqref{homog} there.  

Once such an initial multigrid hierarchy has been constructed, 
the current sets of test vectors are updated with approximations
of the near kernel that are computed using a bootstrap multigrid eigensolver
based on the existing multigrid structure.
The bootstrap multigrid eigensolver begins by
computing the $k_e$ left and right singular vectors with the smallest
singular values of a generalized SVD for the coarsest grid  
operator, $D_L$.
More specifically, define the composite restriction and interpolation operators for $
l=2,...,L$ by
\begin{eqnarray*}\label{composRnP}
P_{l} &=& P_{2}^{1}\cdot \ldots \cdot P_{l}^{l-1},\\
R_{l} &=&  R_{l-1}^{l}\cdot \ldots \cdot R_{1}^{2},
\end{eqnarray*} 
and correspondingly the coarse-grid operators and associated mass matrices by $D_l =
R_l D P_l$,  $Q_l = R_l R_l^H$, and $T_l = P_l^{H}P_l$. The triplets
$(\widetilde{\sigma}_{L}^{(\kappa)},
\widetilde{u}_{L}^{(\kappa)},\widetilde{v}_{L}^{(\kappa)}), \kappa =
1, ..., k_e$, corresponding 
to the $k_e$ smallest singular values of the coarsest-grid system are
then computed by solving the generalized singular value problem
\begin{eqnarray}
D_L \widetilde{v}_L^{(\kappa)}& &= \widetilde{\sigma}_L^{(\kappa)} Q_L \widetilde{u}_L^{(\kappa)}, \label{SVDRnP1} \\
D_L^H \widetilde{u}_L^{(\kappa)}& &= \widetilde{\sigma}_L^{(\kappa)} T_L \widetilde{v}_L^{(\kappa)}. \label{SVDRnP2}
\end{eqnarray} 
We note that, since the size of the coarsest-grid system matrix $D_L$ is small, these triplets 
can be computed directly by solving  
the equivalent generalized Hermitian (indefinite) eigenvalue problem on the
coarsest grid given by
\begin{equation}\label{eq:cSCD}
\begin{pmatrix}
      & D_L    \\
    D_L^H & 
\end{pmatrix}
 \begin{pmatrix}
     \widetilde{U} & \widetilde{U}  \\    
    \widetilde{V}   & - \widetilde{V}
\end{pmatrix}  = 
  \begin{pmatrix}
  Q_L   &     \\
    &  T_L
\end{pmatrix} \begin{pmatrix}
     \widetilde{U} & \widetilde{U}  \\    
    \widetilde{V}  & - \widetilde{V}
    \end{pmatrix}
\begin{pmatrix}
 \widetilde{\Sigma} & \\
  & -\widetilde{\Sigma}
\end{pmatrix}
 , 
 \end{equation} 
where the diagonal entries of $\widetilde{\Sigma}$ contain the ordered
approximate singular values. Note, that~\eqref{eq:cSCD} is obtained
from~\eqref{eq:symev} by a Galerkin construction using the interpolation operator
\begin{equation*}
  \widehat{P}_{L} =
  \begin{pmatrix}
    R_{L} & \\
    & P_{L}
  \end{pmatrix}.
\end{equation*}
The following observation guides the construction of the bootstrap
multigrid eigensolver: If $(\widetilde{\sigma}_{l}^{(\kappa)},
\widetilde{u}_{l}^{(\kappa)},\widetilde{v}_{l}^{(\kappa)})$ is a
triplet of the finer grid SVD and if there exist coarse-grid vectors
$\widetilde{u}_{l+1}^{(\kappa)}$ and
$\widetilde{v}_{l+1}^{(\kappa)}$ such that
$\widetilde{u}_{l}^{(\kappa)}  = R_{l+1}^{l}
\widetilde{u}_{l+1}^{(\kappa)}$ and $\widetilde{v}_{l}^{(\kappa)} =
P_{l+1}^{l} \widetilde{v}_{l+1}^{(\kappa)} $, then
 $(\widetilde{\sigma}_{l+1}^{(\kappa)},
 \widetilde{u}_{l+1}^{(\kappa)},\widetilde{v}_{l+1}^{(\kappa)})$ is a
 triplet on the coarse-grid, i.e,
\begin{equation*}
  R^{l+1}_{l}D_{l}P_{l+1}^{l} \widetilde{v}_{l+1}^{(\kappa)} = \widetilde{\sigma}_{l+1}^{(\kappa)}
 Q_{l} \widetilde{u}_{l+1}^{(\kappa)}, 
 \end{equation*} 
 and, with $P^{l+1}_{l}:= (P^{l}_{l+1})^{H}$,
\begin{equation*}
P^{l+1}_{l}D^H_{l}R_{l+1}^{l} \widetilde{u}_{l+1}^{(\kappa)} = \widetilde{\sigma}_{l+1}^{(\kappa)}
  T_{l} \widetilde{v}_{l+1}^{(\kappa)}.
\end{equation*}
This result gives a relation among the singular triplets
computed in the bootstrap setup on all grids, which we now use
to derive a multigrid eigensolver for the Hermitian
system~\eqref{eq:symev}.  
On finer grids, starting with $l = L-1$, we define a smoother for the systems
\begin{eqnarray}
D_{l}\widetilde{v}_{l}^{(\kappa)}& &= \widetilde{\sigma}_{l}^{(\kappa)} Q_{l} \widetilde{u}_{l}^{(\kappa)}, \label{SVDRnP3}\\
D_{l}^H\widetilde{u}_{l}^{(\kappa)}& &=
\widetilde{\sigma}_{l}^{(\kappa)} T_{l} \widetilde{v}_{l}^{(\kappa)}, \label{SVDRnP4}
\end{eqnarray} 
by a scheme that applies the Kaczmarz iteration 
to each of these two systems separately, alternating between the
two. 
Here, the singular value approximations
$\widetilde{\sigma}_{l}^{(\kappa)} $ are 
updated after each such alternating sweep as
\begin{eqnarray}\label{eq:svals}
\displaystyle{\widetilde{\sigma}_{l}^{(\kappa)} =   \frac{\langle D_{l} \widetilde{v}_{l}^{(\kappa)},\widetilde{u}_{l}^{(\kappa)}\rangle }{\langle Q_{l}\widetilde{u}_{l}^{(\kappa)} ,\widetilde{u}_{l}^{(\kappa)}\rangle^\frac12
\langle T_{l} \widetilde{v}_{l}^{(\kappa)},\widetilde{v}_{l}^{(\kappa)}\rangle^\frac12 },  \quad \kappa = 1, \hdots , k_e. }
\end{eqnarray}
After several such sweeps, the resulting approximations are normalized with respect to the mass matrices
$$
\widetilde{v}_{l}^{(\kappa)} = \frac{\widetilde{v}_{l}^{(\kappa)}}{ \langle T_{l} \widetilde{v}_{l}^{(\kappa)},\widetilde{v}_{l}^{(\kappa)}\rangle^\frac12} \quad \mbox{
and} \quad 
\widetilde{u}_{l}^{(\kappa)} =  \frac{\widetilde{u}_{l}^{(\kappa)} }{\langle Q_{l}\widetilde{u}_{l}^{(\kappa)} ,\widetilde{u}_{l}^{(\kappa)}\rangle^\frac12}
\; .$$
They are then added to the sets of right and left test vectors
\begin{eqnarray*}
\tvU_l = \tvU_l^r \cup \tvU_l^e \quad &&\text{with} \quad \tvU^e_l := \{\widetilde{u}^{(\kappa)}_{l}, \: \kappa = 1,\ldots,k_e\}, \\
\tvV_l =  \tvV^r_l \cup \tvV^e_l \quad &&\text{with} \quad \tvV^e_l:= \{\widetilde{v}^{(\kappa)}_{l}, \: \kappa = 1,\ldots,k_e\},
\end{eqnarray*} 
which are used to recompute $R_{l+1}^{l}$ and $P_{l+1}^{l}$, respectively, using the least squares process.
Here, the superscripts $r$ and $e$ are used to distinguish between the sets of test vectors resulting
from applying Kaczmarz relaxation to the homogenous problems \eqref{homog}
and the sets of test vectors coming from 
the Kaczmarz iterations applied to \eqref{SVDRnP3} and \eqref{SVDRnP4}, with initial guesses coming from 
solutions to the coarsest-grid eigenproblem~\eqref{eq:cSCD}. 

\subsection{The $\gamma_5$-symmetry and Galerkin coarsening}
A Galerkin coarsening scheme for the non-Hermitian Wilson matrix was first considered
in the context of an adaptive aggregation-based multigrid solver in~\cite{Bran_PRL_2010}. 
The idea was motivated by the $\gamma_5$-symmetry of the Wilson
 matrix. Recall that, by~\eqref{eq:rightevleftev} for each  
 eigenpair $(\lambda,v_\lambda)$ there corresponds
a left eigenpair $(\bar{\lambda},\Gamma_5 v_\lambda)$. 
This motivates the choice $R_{1}^{2} = (\Gamma_{5}P_{2}^{1})^{H}$ on
the finest grid.  Maintaining a similar relation $R^{l+1}_{l}
= (\Gamma_{5,l}P_{l+1}^{l})^H$ on all coarser grids means that
the coarse grid system should satisfy
\begin{equation*}
  \Gamma_{5,l+1} D_{l+1} = D_{l+1}^H\Gamma_{5,l+1},
\end{equation*} with $\Gamma_{5,l+1}$ a unitary and Hermitian matrix
inherited from $\Gamma_{5,l}$ on grid $l$. This is achieved in a
Galerkin approach, i.e., $D_{l+1} = (P_{l+1}^{l})^HD_{l}P_{l+1}^{l}$,
if interpolation satisfies
\begin{equation}\label{eq:gammaC}
  \Gamma_{5,l} P_{l+1}^{l}  =
 P_{l+1}^{l} \Gamma_{5,l+1},
\end{equation} which, in turn, is fulfilled if we enforce that each
variable only interpolates from variables of the same spin. In other words, 
we fix the sparsity of interpolation according to the spin ordering of $\Gamma_{5,l}$ and 
$\Gamma_{5,l+1}$ as
\begin{equation}\label{eq:bs}
  P_{l+1}^l = 
  \begin{pmatrix}
    \ast  &    \\
    & \ast 
  \end{pmatrix},
\end{equation} which gives $\Gamma_{5,l} = \gamma_{5} \otimes I_{n_{l}}$.
In the context of full coarsening of the grid, as defined in Section \ref{sec:coarsening}, 
for each grid point $i\in\finevar$ we obtain the structure in~\eqref{eq:bs}
if we define the interpolatory
sets for the two spin variables $i_1$ and $i_2$ at grid point $i$ independently as
\begin{equation}\label{eq:Csubi}
\coarsevar_{i_*} = \bigg\{ j_* \in \coarsevar \: \bigg| \:
D_{i_*,j_*} \neq 0 \bigg\}.
\end{equation}
In this way, a given spin variable $i_*$ at grid point $i$ interpolates only from variables $j_*$ 
of the same spin defined at neighboring grid points $j$.  
Using these assumptions, direct computation shows that the
Petrov Galerkin coarse-grid correction reduces to a Galerkin
correction for $D$ (cf.~\cite{Bran_PRL_2010}), i.e., for $R^{l+1}_{l}
= (\Gamma_{5,l}P_{l+1}^{l})^H$ we have
\begin{equation*}
  P_{l+1}^{l}\left(R^{l+1}_{l}D_{l}P_{l+1}^{l}\right)^{-1}R^{l+1}_{l}
=  P_{l+1}^{l}\left(\left(P^{l}_{l+1}\right)^{H}D_{l}P_{l+1}^{l}\right)^{-1}\left(P^{l}_{l+1}\right)^{H}.
\end{equation*}

This same structure preserving form of $P$ and the $\gamma_{5}$-symmetry of $D$ 
imply, in addition, the equivalence of the Petrov Galerkin
bootstrap AMG setup for $D$ and a Galerkin approach for the Hermitian indefinite
matrix $Z = \Gamma_{5}D$ as we are going to show next.

Indeed, since $Z$ is Hermitian there exists a unitary matrix $V \in \linspace[n\times
n]{C}$ such that 
\begin{equation*}\label{eq:ev_sv}
Z = \Gamma_5D = V\Lambda V^H
\end{equation*} with $\Lambda =
\operatorname{diag}(\lambda_{1},\ldots,\lambda_{n})$, which directly implies
\begin{equation}\label{eq:ev_sv2}
  D = \Gamma_5 V \operatorname{sign}(\Lambda) |\Lambda| V^H.
\end{equation}
Thus, $U = \Gamma_5 V \operatorname{sign}(\Lambda)$, $V = V$ are
unitary and with $S = |\Lambda|$ we have found an expression for the
SVD of $D$ in terms of eigenvectors and eigenvalues of $Z$. 
This relation implies that the space spanned by any pair of right and
left singular vectors $u_{i}, v_{i}$ satisfies
\begin{equation*}\label{eq:gammaUV}
\operatorname{span}(u_{i}) = \operatorname{span}(\Gamma_5 v_{i}),
\end{equation*} since the factor $\operatorname{sign}(\lambda_{i})$ does
not change the $\operatorname{span}$. 
This result is also true for any pair of subspaces spanned by
a collection of pairs of singular vectors. As a consequence, the choice $R^{l+1}_{l}
= (\Gamma_{5,l}P_{l+1}^{l})^H$, which was motivated in~\cite{Bran_PRL_2010} by the
correspondence between left and right eigenvectors with respect to the
$\gamma_5$-symmetry, is now justified in
terms of left and right singular vectors, assuming that $P_{l+1}^{l}$ is
constructed from test vectors approximating right singular vectors to small
singular values.
Additionally, we have the following new results relating the setup for
$D$ and $Z$.

\begin{theorem}\label{thm:svd2eig}
Assume that $R^{l+1}_{l} = (\Gamma_{5,l}P_{l+1}^{l})^H$ and $P^{l}_{l+1}$ has the block structure defined in~\eqref{eq:bs} such that 
$\Gamma_{5,l}P_{l+1}^{l} = P_{l+1}^{l}\Gamma_{5,l+1}$.
Then, the two equations~\eqref{SVDRnP3} and~\eqref{SVDRnP4} are equivalent to  
$$Z_l \widetilde{v}_{l}^{(\kappa)} = \widetilde{\lambda}_{l}^{(\kappa)} T_{l} \widetilde{v}_{l}^{(\kappa)}, \quad \mbox{with}  \quad \widetilde{u}_{l}^{(\kappa)} = \Gamma_{5,l} \widetilde{v}_{l}^{(\kappa)},$$
where $Z_l = \Gamma_{5,l} D_l.$
Thus, in particular, we have the following equivalences.
\begin{enumerate}[$(i)$]
\item
The singular-value problem on the coarsest grid given
by~\eqref{SVDRnP1} and~\eqref{SVDRnP2} is equivalent to the generalized eigenproblem
\begin{equation*}\label{eq:eigQ}
Z_L \widetilde{V} = T_L \widetilde{V} \widetilde{\Lambda}, \text{\
  where\ } Z_L = \Gamma_{5,L} D_L \widetilde{V} .
\end{equation*} 
\item
Kaczmarz relaxation applied to either of the equations~\eqref{SVDRnP3}
or~\eqref{SVDRnP4} reduces to applying Kaczmarz to the equation
\begin{equation*}\label{eq:SVDV}
Z_l \widetilde{v}_{l}^{(\kappa)} = \widetilde{\lambda}_{l}^{(\kappa)} T_{l} \widetilde{v}_{l}^{(\kappa)}.
\end{equation*} 
More precisely, the correction $s_i$ used in the Kaczmarz updates for
the systems with $D_l$, defined via the equation
\begin{equation*}\label{eq:kaczsvd}
  \left\langle D_{l}^H \left[\widetilde{\lambda}_{l}^{(\kappa)} T_{l} \Gamma_{5,l}\widetilde{v}_{l}^{(\kappa)}  - D_{l} (\widetilde{v}_{l}^{(\kappa)}+ s_i e_i)\right], e_i \right\rangle = 0,
\end{equation*} 
can equivalently be computed using the equation for $s_i$ in terms of $Z_l$ given by
\begin{equation*}\label{eq:kaczeig}
\left\langle Z_{l} \left[\widetilde{\lambda}_{l}^{(\kappa)} T_{l}\widetilde{v}_{l}^{(\kappa)}  - Z_{l} (\widetilde{v}_{l}^{(\kappa)}+ s_i e_i)\right], e_i \right\rangle = 0.
\end{equation*}
\item The singular value approximations defined in~\eqref{eq:svals}
  satisfy $\widetilde{\sigma}_{l}^{(\kappa)} =
  |\widetilde{\lambda}_{l}^{(\kappa)} |$, where
  $\widetilde{\lambda}_{l}^{(\kappa)}$ is the Ritz value
\begin{equation}\label{eqevZ}
\displaystyle{\widetilde{\lambda}_{l}^{(\kappa)} 
= 
 \frac{\langle Z_l \widetilde{v}_{l}^{(\kappa)},\widetilde{v}_{l}^{(\kappa)}\rangle }{
\langle T_{l} \widetilde{v}_{l}^{(\kappa)},\widetilde{v}_{l}^{(\kappa)}\rangle }. }
\end{equation}
\end{enumerate}
\end{theorem}
\begin{proof}
Since $R^{l+1}_{l} = (\Gamma_{5,l}P_{l+1}^{l})^H$ we have $Q_l = T_l$.
Now, using \eqref{eq:ev_sv2} 
we have
$$ \widetilde{u}_{l}^{(\kappa)} = \operatorname{sign}\left(\widetilde{\lambda}_{l}^{(\kappa)}\right) \Gamma_{5,l} \widetilde{v}_{l}^{(\kappa)} \quad \text{and} \quad
\widetilde{\sigma}_{l}^{(\kappa)} =\operatorname{sign}\left(\widetilde{\lambda}_{l}^{(\kappa)}\right)\widetilde{\lambda}_{l}^{(\kappa)}.$$
Parts $(i)$--$(iii)$ now follow by substitution.
\end{proof}

This theorem implies that the overall Petrov Galerkin bootstrap AMG
setup process developed for the non-Hermitian Wilson matrix $D$ in
Section \ref{sec:bootamg} is equivalent to a Galerkin setup process for the Hermitian form of the Wilson matrix $Z$.
We provide additional details of the Galerkin bootstrap setup and multigrid eigensolver for $Z$ in Figure~\ref{fig:boot:setupcycle}.

\begin{figure}
\begin{center}
     \tikzstyle{greenpoint}=[circle,inner sep=0pt,minimum size=2mm,draw=black!100,fill=black!30]
    \tikzstyle{whitepoint}=[rectangle,inner sep=0pt,minimum size=2mm,draw=black!100,fill=black!80]
    \tikzstyle{redpoint}=[diamond,inner sep=0pt,minimum size=2.55mm,draw=black!100,fill=black!80]
    \tikzstyle{blackpoint}=[circle,inner sep=0pt,minimum size=2mm,draw=black!100,fill=black!100]
    \tikzstyle{bluepoint}=[circle,inner sep=0pt,minimum size=2mm,draw=black!100,fill=black!0]
    \resizebox{\textwidth}{!}{\begin{tikzpicture}
      \draw [sharp corners] (0,0) node[blackpoint] {} --
      ++(300:1cm) node[blackpoint] {} --
      ++(300:1cm) node[blackpoint] {} --
      ++(300:1cm) node[blackpoint] {} --
      ++(300:1cm) node[redpoint] (L4) {} --
      ++(60:1cm) node[bluepoint] (L3) {} --
      ++(60:1cm) node[bluepoint] (L2) {} --
      ++(60:1cm) node[bluepoint] (L1) {} --
      ++(60:1cm) node[whitepoint] (L0) {} --
      ++(0:1cm) node[greenpoint] {} --
      ++(300:1cm) node[greenpoint] (end) {};
      \draw[dashed] (end) --  ++(300:1cm);
      
      \draw (L0) + (3cm,0) node[blackpoint,label=right:{\small Relax on $Zv=0, v \in \tvV^r$, compute $P$}] {};
      \draw (L1) + (3.5cm,0) node[redpoint,label=right:{\small Compute $v$, s.t., $Zv=\lambda Tv$, update $\tvV^e$}] {};
      \draw (L2) + (4.0cm,0) node[bluepoint,label=right:{\small Relax on $Zv = \lambda T v, v \in \tvV^e$}] {};
      \draw (L3) + (4.5cm,0) node[greenpoint,label=right:{\small Relax on $Zv=0, v \in \tvV^r$ and $Zv = \lambda T v, v \in \tvV^e$, recompute $P$}] {};
      \draw (L4) + (5cm,0) node[whitepoint,label=right:{\small Test
        MG method, update $\tvV$}] {};

      \draw (L3) +(0,4cm) node (start) {};
      \draw[sharp corners] (start) ++(300:-3cm) node[blackpoint] {} --
      ++(300:1cm) node[blackpoint] {} --
      ++(300:1cm) node[blackpoint] {} --
      ++(300:1cm) node[redpoint] {} --
      ++(60:1cm) node[greenpoint] {} --
      ++(300:1cm) node[redpoint] {} --
      ++(60:1cm) node[bluepoint] {} --
      ++(60:1cm) node[greenpoint] {} --
      ++(300:1cm) node[greenpoint] {} --
      ++(300:1cm) node[redpoint] {} --
      ++(60:1cm) node[greenpoint] {} --
      ++(300:1cm) node[redpoint] {} --
      ++(60:1cm) node[bluepoint] {} --
      ++(60:1cm) node[greenpoint] {} --
      ++(300:1cm) node[greenpoint] {} --
      ++(300:1cm) node[redpoint] {} --
      ++(60:1cm) node[bluepoint] {} --
      ++(60:1cm) node[bluepoint] {} --
      ++(60:1cm) node[whitepoint] {} -- 
      ++(0:1cm) node[greenpoint] {} --     
      ++(300:1cm) node[greenpoint] {} --
      ++(300:1cm) node[greenpoint] {} --
      ++(300:1cm) node[redpoint] {} --
      ++(60:1cm) node[greenpoint] (end) {};
      \draw[dashed] (end) --  ++(300:1cm);
    \end{tikzpicture}}
  \caption{Galerkin Bootstrap AMG W cycle and V cycle setup schemes.\label{fig:boot:setupcycle}}
\end{center}
\end{figure}
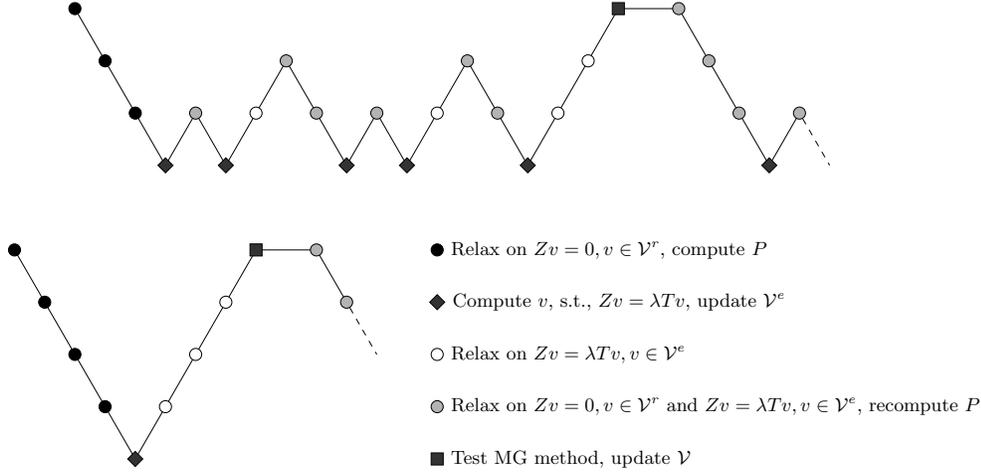

\section{Numerical results}\label{sec:numres}

In this section, we present numerical tests of our Ga\-ler\-kin bootstrap
AMG setup algorithm for the Wilson discretization of the Dirac
equation. We apply our method to the Schur complement system resulting
from an odd-even reduction of the Wilson matrix: $$\widehat{D}\psi_{e} = b_{e} - D_{eo}b_{o},$$ with
$\widehat{D}$ defined as in \eqref{eq:wilsonSchur}, i.e., $D_{1} := \widehat{D}$.  In all tests, the
sets of coarse variables are defined by full coarsening (see Figure~\ref{fig:oddevencoarsening}).  
We use nearest neighbor interpolation defined in terms of the graph of the matrix $\widehat{D}$, as
in \eqref{eq:Csubi} which preserves the spin-structure of the system. 
The maximal number of interpolatory points is
bounded by four and the Galerkin coarse-grid operator has the same
sparsity pattern on all grids, with at most 18 non-zeros per
row and, thus, the grid and operator complexities (cf.~\cite{Trottenberg:2000:MUL:374106}) are bounded by $1.4$
independent of the problem size.

Fixing the coarsening and sparsity pattern of $P$ as such, we study the performance of the bootstrap algorithm applied 
to the Hermitian and indefinite matrix $\Gamma_e \widehat{D}$ in \eqref{eqevZ2} for $N=128$ and $256$ with $\beta = 3,6,10$ in \eqref{eq:wilsongauge} and various choices of the shifts, $m$, used in setting the minimal eigenvalue of the Wilson matrix $D = mI +D_0$.  
For completeness, we consider nine distinct gauge field configurations 
for each pair of parameters $m$ and $\beta$.  
In the plots, the (dashed) lines from
light to dark correspond to increasing configuration numbers, 
from 11,000, ..., 19,000, respectively.

The weighted least squares approach in~\eqref{eq:LSfuncrowi} is used
to define the entries of the interpolation operators.  We set the
number of test vectors computed using relaxation as $k_{r}=|\tvV_{r}|
=8$ and the number of additional eigenvector approximations computed
in the multigrid eigensolver as $k_{e}=|\tvV_{e}|=8$.   The least
squares form of $P$ is then computed on each grid using the combined
sets of up to $k = k_{r} + k_{e}$ test vectors (the initial hierarchy
is constructed using only $k_{r}$ test vectors).  On the finest grid, the vectors,
$v^{(1)}, \ldots , v^{(k_r)}$, used to initialize the bootstrap process,
are generated randomly and independently with a normal distribution with expectation
zero and variance one. 

We use a $W(4,4)$ cycle solve phase with Kaczmarz smoothing.  The
problem is coarsened to a coarsest-grid system with $N=16$, which is
solved directly, giving 4- and 5-grid methods for the $N=128$ and
$N=256$ problem sizes, respectively.  The reported estimates of asymptotic
convergence rates, $\rho$, of the resulting solver are computed by
\begin{equation*}
\rho = \frac{\|e^{\nu}\|}{\|e^{\nu-1}\|},
\end{equation*}
where $e^\nu$ denotes the error after $\nu$ multigrid iterations,
i.e., the asymptotic convergence rate is measured upon convergence to
the specified tolerance or after $\nu = 100$ iterations.
The number of BAMG preconditioned GMRES($32$) iterations needed to
reduce the $\ell_2$ norm of the relative residual to this same
tolerance is also reported.

\subsection{The 2D Wilson matrix -- Bootstrap W cycle setup
}\label{sec:bamgWres}
In our first set of tests, we use a W(10,10) cycle with Kaczmarz
relaxation in the first bootstrap setup cycle.  Then, after an
intermediate adaptive step in which we apply two W(4,4) cycles to update the
test vector with the smallest value of $|\widetilde{\lambda}_{0}^{(\kappa)}|$ in \eqref{eqevZ} on the finest grid only, we apply a second W(5,5) setup
cycle to update the sets of test vectors on all grids which are used to compute 
the final multigrid hierarchy.   
After extensive testing of the proposed setup approach, we found these
settings to yield a robust solver for all test problems considered.  A
few remarks regarding these choices of the settings for the setup
algorithm are in order before presenting the results of these
experiments. 
 
First, we note that the extra smoothing steps are applied in the
initial W cycle since we have observed in practice that this gives a
sufficiently accurate initial hierarchy from which the solver can then
be constructed.  Generally, using fewer relaxation steps and a larger
number of bootstrap setup cycles is less efficient than an approach in which
more relaxation steps are used in each of the cycles, so that fewer cycles are
needed to obtain a suitable solver.  
Moreover, for the highly ill-conditioned Wilson matrices 
we consider, we find that at least two bootstrap cycles with one intermediate adaptive step
is needed in order to obtain an efficient solver for all test
problems, unless we increase the number of relaxations that are used in the initial bootstrap cycle significantly.  

 Additionally, we mention that the intermediate adaptive step is applied only to a single test vector, namely,
 the one that yields the smallest value of $|\widetilde{\lambda}_{0}^{(\kappa)}|$ in \eqref{eqevZ} .  
 While we use this step in all
 tests, we have observed that it turns out to be beneficial mostly in cases where the
 shift is chosen such that $\eta_{min}(D)$ in~\eqref{eq:etamin} is almost
 zero, i.e., for the most ill-conditioned cases.  In such cases, this simple
 modification to the algorithm reduces the number of bootstrap 
 setup cycles needed to obtain an efficient solver by at least one and in most
 cases two or more, assuming the number of smoothing steps are not increased.   

 As a final remark, we comment that the use of W cycles as opposed to
 V cycles in the setup and solve phases of the algorithm is needed to
 compensate for the fact that the problem is coarsened to $N=16$ and
 that the maximum number of interpolation variables is limited to four.  We
  observed that for certain realizations of the gauge fields, the resulting
  coarsest spaces as defined in our algorithm are too lean to compensate
  for the large number of near kernel vectors that they have to
  approximate.  An alternative strategy, that we did not explore here,
  would be to increase the number of interpolation points on
  coarser grids.

\begin{figure}
  \begin{center}
    \subfigure[$N=128$, $\beta = 3$ \label{fig:boot:Wresults1283}]{
      \includegraphics[width=.9\textwidth]{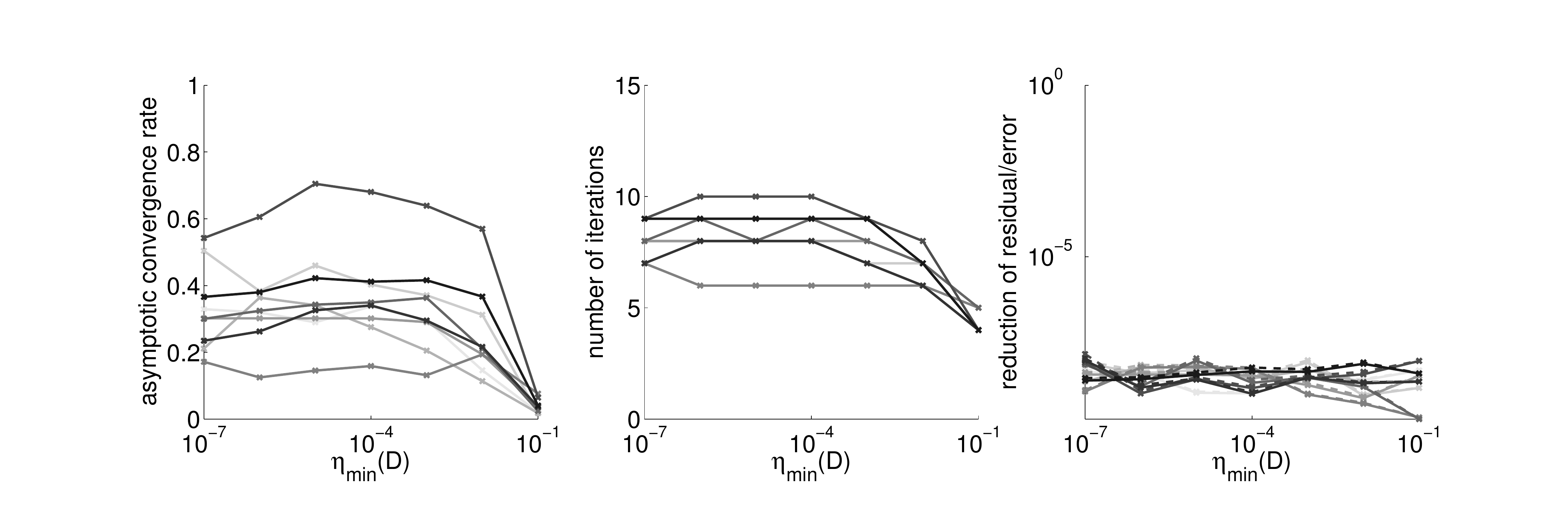}}
    \subfigure[$N=128$, $\beta = 6$ \label{fig:boot:Wresults1286}]{
      \includegraphics[width=.9\textwidth]{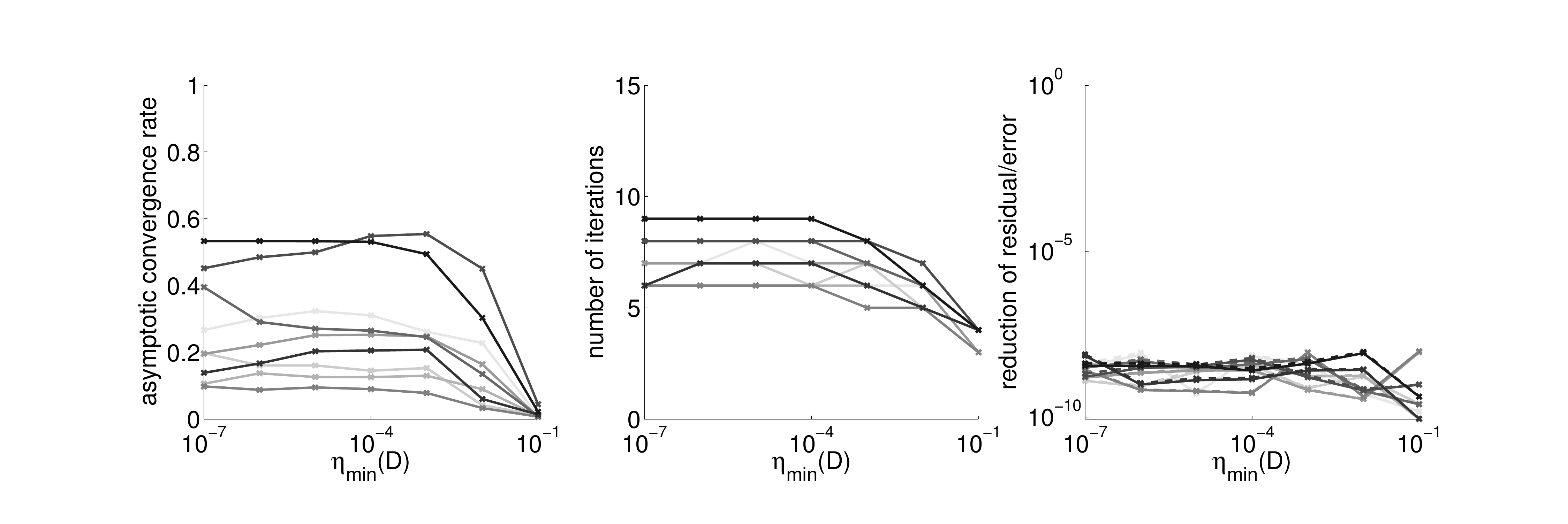}}
    \subfigure[$N=128$, $\beta = 10$ \label{fig:boot:Wresults12810}]{
      \includegraphics[width=.9\textwidth]{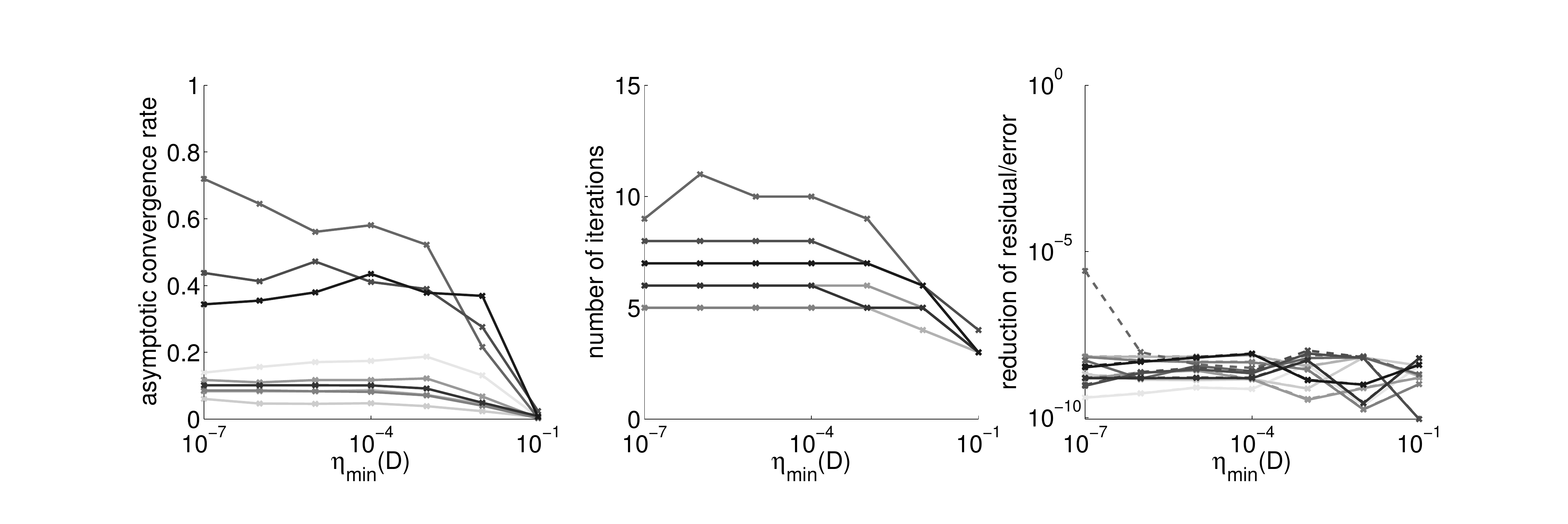}}
  \end{center}        
  \caption{Results of bootstrap AMG and bootstrap AMG preconditioned GMRES($32$) applied to the 
    odd-even reduced matrix for $N=128$ and different values of
    $\beta$. The results for the 
    light to dark lines
    correspond to different gauge field configurations, going from light to dark with
    increasing configuration number. 
    On the left, plots of the estimates of the convergence rates $\rho$
    for the stand-alone solver versus
    $\eta_{\min}(D)$ defined in~\eqref{eq:etamin}, corresponding to different
    diagonal shifts $m$, are provided. 
    In the middle, the number of bootstrap AMG  preconditioned
    GMRES($32$) iterations needed to reduce the $\ell_2$ norm of the
    relative residual by a factor of $10^{-8}$ 
    is plotted against $\eta_{\min}(D)$.  The plots on the right
    contain the $\ell_2$ norms of the relative residuals (solid lines) and relative errors (dashed lines)
    computed using the resulting solution versus $\eta_{\min}(D)$.
    \label{fig:boot:Wresults128}} 
\end{figure}  

\begin{figure}
\begin{center}
  \subfigure[$N=256$, $\beta = 3$ \label{fig:boot:Wresults2563}]{
        \includegraphics[width=.9\textwidth]{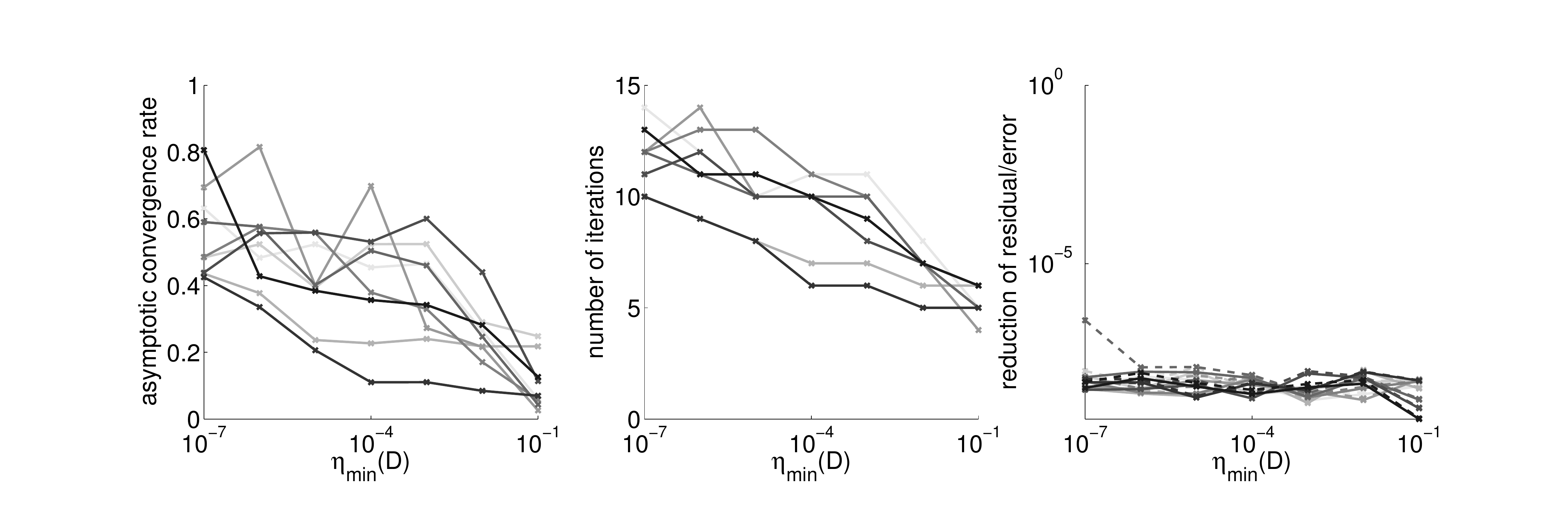}}
 \subfigure[$N=256$, $\beta = 6$ \label{fig:boot:Wresults2566}]{
        \includegraphics[width=.9\textwidth]{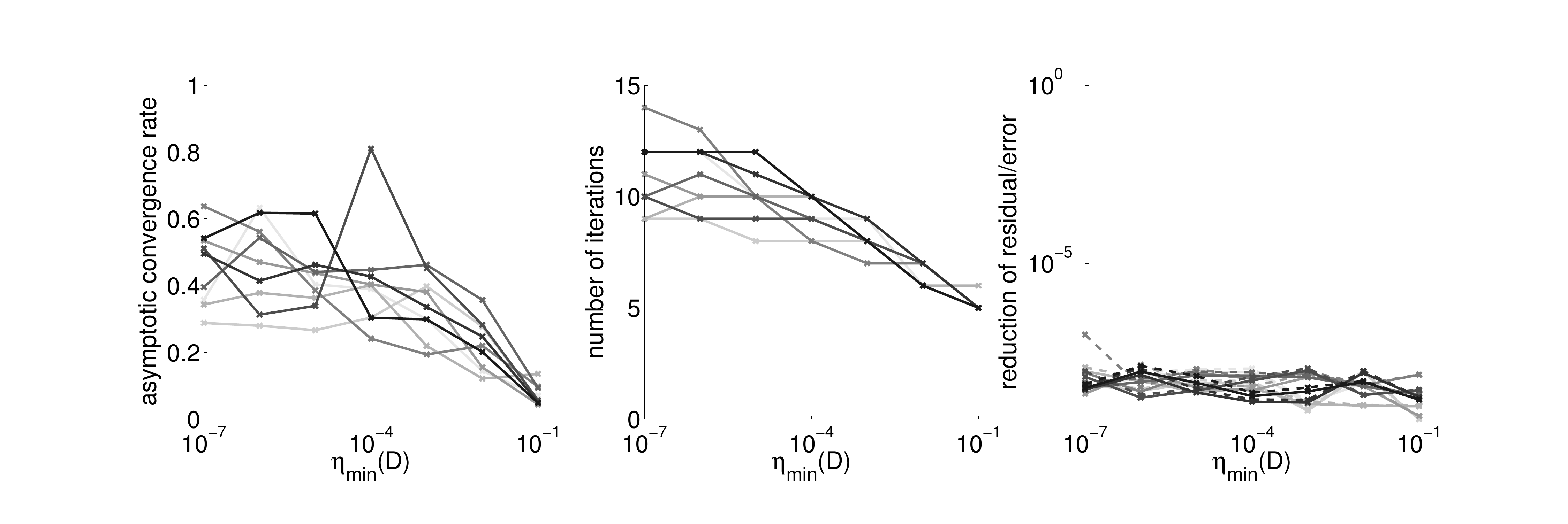}}
 \subfigure[$N=256$, $\beta = 10$ \label{fig:boot:Wresults25610}]{
        \includegraphics[width=.9\textwidth]{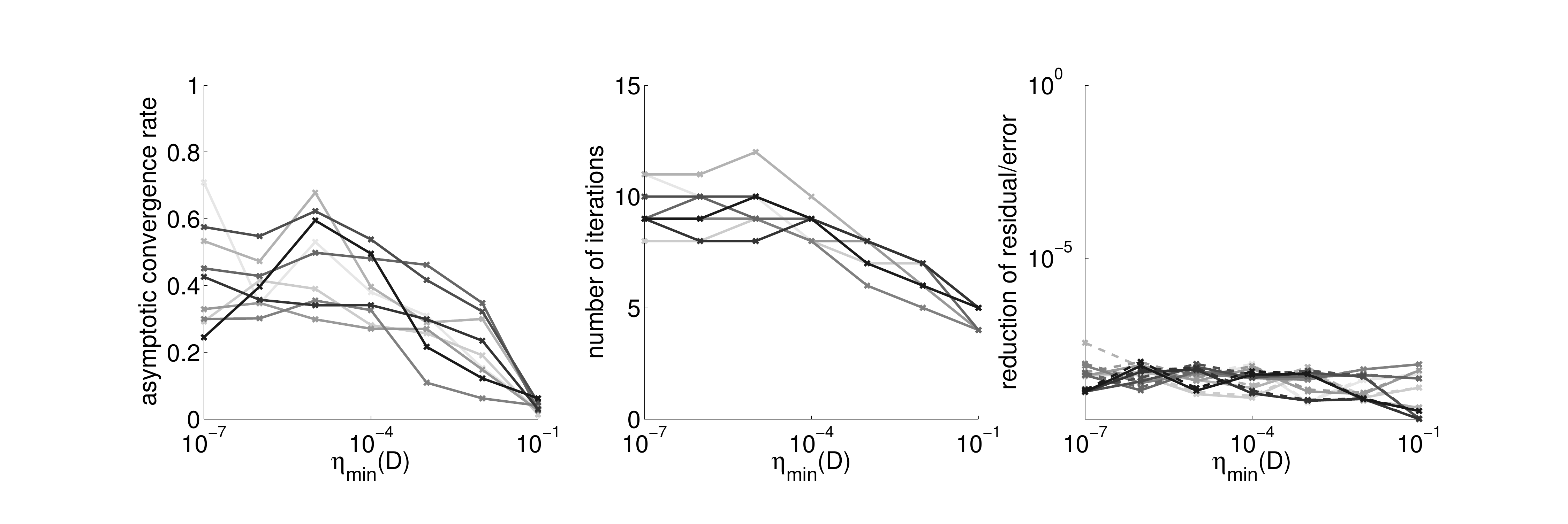}}
  \end{center}        
  \caption{Results of bootstrap AMG and bootstrap AMG preconditioned GMRES($32$) applied to the 
    odd-even reduced matrix for $N=256$ and different values of
    $\beta$. The results for the light to dark lines
    correspond to different gauge field configurations, going from light to dark with
    increasing configuration number. 
    On the left, plots of the estimates of the convergence rates $\rho$
    for the stand-alone solver versus
    $\eta_{\min}(D)$ defined in~\eqref{eq:etamin}, corresponding to different
    diagonal shifts $m$, are provided.  In the middle, the number of bootstrap AMG  preconditioned
    GMRES($32$) iterations needed to reduce the $\ell_2$ norm of the
    relative residual by a factor of $10^{-8}$ 
    is plotted against $\eta_{\min}(D)$.  The plots on the right
    contain the $\ell_2$ norms of the relative residuals (solid lines) and relative errors (dashed lines)
    computed using the resulting solution versus $\eta_{\min}(D)$.
\label{fig:boot:Wresults256}} 
\end{figure}  

The results of these experiments for $N=128$ are reported in Figure~\ref{fig:boot:Wresults128}
and the ones in Figure~\ref{fig:boot:Wresults256} are for $N=256$.
For both problem sizes we report results for $\beta=3,6,10$ for
nine different gauge field configurations.  
Here, we see that for both problem sizes, the stand-alone solver is convergent and only in few exceptional cases do these rates
go above $0.6$. In addition, the number of preconditioned GMRES($32$) iterations is also fairly uniform for fixed $N$ and varying values of
$\beta$, the minimum eigenvalue of $D$, and different
configurations.  Further, although the number of preconditioned GMRES($32$) iterations
seems to grow slightly for the larger problem sizes, we see that the number of iterations is reduced by roughly two orders of magnitude when
compared to the number of iterations needed by CGNR without preconditioning, as reported in Figure~\ref{CGNodd} for the odd-even system and in all tests the outer preconditioned GMRES($32$) method never
reaches a restart.
Finally, we observe that in almost all cases the errors and residuals are
within an order of magnitude, which further demonstrates the effectiveness of the proposed method.

\subsection{The 2D Wilson matrix -- Bootstrap super-V cycle setup
}
\begin{figure}
\begin{center}
  \includegraphics[width=.9\textwidth]{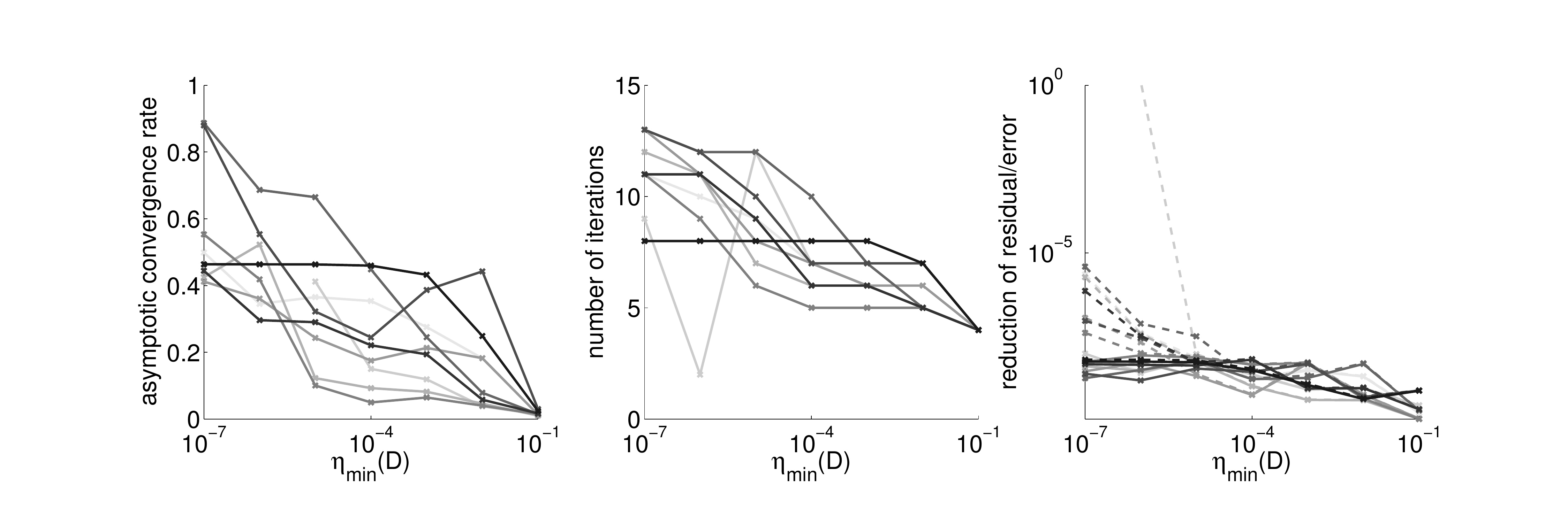}
       \end{center}        
  \caption{Results of super-V cycle setup and W cycle bootstrap AMG and bootstrap AMG preconditioned GMRES($32$) solvers applied to the 
    odd-even reduced matrix for $N=128$ and $\beta =6$. The results for the 
    light to dark lines
    correspond to different gauge field configurations, going from light to dark with
    increasing configuration number. 
    On the left, plots of the estimates of the convergence rates $\rho$
    for the stand-alone solver versus
    $\eta_{\min}(D)$ defined in~\eqref{eq:etamin}, corresponding to different
    diagonal shifts $m$, are provided. 
    In the middle, the number of bootstrap AMG  preconditioned
    GMRES($32$) iterations needed to reduce the $\ell_2$ norm of the
    relative residual by a factor of $10^{-8}$ 
    is plotted against $\eta_{\min}(D)$.  The plots on the right
    contain the $\ell_2$ norms of the relative residuals (solid lines) and relative errors (dashed lines)
    computed using the resulting solution versus $\eta_{\min}(D)$.
    \label{fig:boot:SVresults}} 
\end{figure}  

In this section, to test the impact of the coarsest-grid eigensolver on the overall effectiveness of the bootstrap AMG
process, we consider a modified setup cycle in which the number of
relaxations used on each grid is the same as in an overall W cycle,
but each grid is visited only once. Specifically, we consider a $V(\mu 2^l,\mu 2^l)$ cycle in the setup phase and as before $W(4,4)$ cycles in
the solve phase. Here, $l$ denotes the given grid and the notation $V(\mu 2^l,\mu 2^l)$
means that we use $\mu 2^l$ pre-smoothing iterations and $\mu 2^l$
post-smoothing iterations on grid $l=1,...,L-1$. 
Thus, the number of smoothing steps applied on each grid is the same as in a W$(\mu,\mu)$ cycle, but the coarsest-grid eigensolve is applied only once.  We use the same settings for the setup cycles as we used in Section \ref{sec:bamgWres}, 
i.e., $k_r = k_e = 8$, $\mu = 10$ in the first super-V cycle setup, and $\mu = 5$ in the second setup cycle.  
We test the method for $N=128$ and $\beta = 6$
and use the same nine gauge field configurations 
we used in the previous section.  

The results of these tests are reported in Figure~\ref{fig:boot:SVresults}.  As we can see in the plots on the left and in the middle in the figure, the stand-alone solver and preconditioner with this setup strategy is less effective than the one coming from the W cycle setup.  Further, we see in the plots on the right that for some of the most ill-conditioned cases, the residuals and errors again differ by several orders of magnitude.  Together, these results demonstrate the effectiveness of the method when the coarsest-grid eigensolver is repeated within the bootstrap setup process, especially in the most ill-conditioned cases.

\section{Concluding Remarks}\label{sec:conclusions}
In this paper, we designed and tested a bootstrap approach for
computing multigrid interpolation operators for the non-Hermitian
Wilson discretization of the Dirac equation. 
As in any efficient multigrid solver, these operators have to be accurate for the near kernel vectors of the problem's 
finest-grid operator. Here, this is achieved by defining interpolation
to fit, in a least squares sense, a set of test vectors that
collectively approximate the near-kernel vectors of the Dirac
matrix.  A main new result is given in Theorem~\ref{thm:svd2eig}, where we used the
$\gamma_5$-symmetry of the Wilson matrix to reduce a Petrov Galerkin
multigrid algorithm for the non-Hermitian Wilson matrix $D$ to a
Galerkin approach for $D$, where the setup process is applied to the
Hermitian and indefinite version $Z = \Gamma_{5}D$ of the Wilson
matrix to derive the solver for $D$ (and $Z$).  

Further, extensive numerical tests have shown that using least squares interpolation together
with a bootstrap AMG setup based on an equivalent Hermitian indefinite
form of the Wilson matrix and an intermediate adaptive setup cycle
leads, in practise, to a robust AMG setup algorithm and thereby solver
and preconditioner for the
Wilson matrix over a wide range of problem parameters. 
All numerical experiments presented in the paper were carried out for
Wilson's discretization of the 2-dimensional Dirac equation on a structured grid, using full
coarsening  and interpolation with a fixed nearest neighbor sparsity
pattern.  This allowed us to concentrate on developing and testing
least squares techniques for computing the interpolation operators and
the impact of the bootstrap AMG setup, the multigrid eigensolver, and
the adaptive step on the efficiency of the resulting multigrid solver.
Generally, we have shown that with a proper choice of these components
of the algorithm, a robust and efficient solver can be constructed.
We note, in addition, that all of the derivations and heuristic arguments used
in formulating the proposed algorithm carry over directly to the 4-dimensional
Wilson matrix arising in Quantum Chromodynamic simulations so that the proposed solver is expected to work well
in this setting, too.

\bibliography{BAMG}

\providecommand{\noopsort}[1]{}
\begin{thebibliography}{10}

\bibitem{Bran_PRL_2010}
R.~Babich, J.~Brannick, R.~C. Brower, M.~A. Clark, T.~A. Manteuffel, S.~F.
  McCormick, J.~C. Osborn, and C.~Rebbi.
\newblock Adaptive multigrid algorithm for the lattice {W}ilson-{D}irac
  operator.
\newblock {\em Phys. Rev. Lett.}, 105:201602, Nov 2010.

\bibitem{Osborne_POS_2009}
Ronald Babich, James Brannick, Richard~C. Brower, Michael~A. Clark, Saul~D.
  Cohen, et~al.
\newblock The {R}ole of multigrid algorithms for {LQCD}.
\newblock {\em PoS}, LAT2009:031, 2009.

\bibitem{BankDupont}
R.~Bank and T.~Dupont.
\newblock A comparison of two multilevel iterative methods for nonsymmetric and
  indefinite elliptic finite element equations.
\newblock {\em SIAM J. Numer. Anal.}, 18:701--718, 1981.

\bibitem{brannick_markov_2011}
M.~Bolten, A.~Brandt, J.~Brannick, A.~Frommer, K.~Kahl, and I.~Livshits.
\newblock Bootstrap {AMG} for {M}arkov chains.
\newblock {\em SIAM J. Sci. Comput.}, 33:3425--3446, 2011.

\bibitem{Bramble_93}
J.~Bramble, D.~Y. Kwak, and J.~Pasciak.
\newblock Uniform convergence of multigrid {V}-cycle iterations for indefinite
  and nonsymmetric problems.
\newblock In {\em Sixth Copper Mountain Conference on Multigrid Methods}, pages
  43--59, 1993.

\bibitem{BAMG2}
A.~Brandt.
\newblock Multiscale scientific computation: review 2001.
\newblock In T.~J. Barth, T.~F. Chan, and R.~Haimes, editors, {\em Multiscale
  and Multiresolution Methods: Theory and Applications}, pages 1--96. Springer,
  Heidelberg, 2001.

\bibitem{BAMG2010}
A.~Brandt, J.~Brannick, K.~Kahl, and I.~Livshits.
\newblock Bootstrap {AMG}.
\newblock {\em SIAM J. Sci. Comput.}, 33:612--632, 2011.

\bibitem{geo}
A.~Brandt, S.~McCormick, and J.~Ruge.
\newblock Algebraic multigrid ({AMG}) for automatic multigrid solution with
  application to geodetic computations.
\newblock Technical report, Colorado State University, Fort Collins, Colorado,
  1983.

\bibitem{oAMG}
A.~Brandt, S.~McCormick, and J.~Ruge.
\newblock Algebraic multigrid ({AMG}) for sparse matrix equations.
\newblock In D.~J. Evans, editor, {\em Sparsity and Its Applications}.
  Cambridge University Press, Cambridge, 1984.

\bibitem{Brannick_Trace_06}
J.~Brannick and L.~Zikatanov.
\newblock Algebraic multigrid methods based on compatible relaxation and energy
  minimization.
\newblock In O.~B. Widlund and D.~E. Keyes, editors, {\em Domain decomposition
  methods in science and engineering XVI}, volume~55, pages 15--26. Springer,
  2007.

\bibitem{MBrezina_2005}
M.~Brezina, R.~Falgout, S.~MacLachlan, T.~Manteuffel, S.~McCormick, and
  J.~Ruge.
\newblock Adaptive amg (a{AMG}).
\newblock {\em SIAM J. Sci. Comput.}, 26:1261--1286, 2005.

\bibitem{gesa}
M.~Brezina, T.~Manteuffel, S.~McCormick, J.~Ruge, and G.~Sanders.
\newblock Towards adaptive smooth aggregation {(aSA)} for nonsymmetric
  problems.
\newblock {\em SIAM J. Sci. Comput.}, 32:4--39, 2010.

\bibitem{Brower:1990ac}
R.~Brower, K.~Moriarty, C.~Rebbi, and E.~Vicari.
\newblock Variational multigrid for nonabelian gauge theory.
\newblock In *Tallahassee 1990, Proceedings, Lattice 90* 89-93. (see HIGH
  ENERGY PHYSICS INDEX 29 (1991) No. 11041).

\bibitem{Brower:1991en}
R.~C. Brower, K.~J.~M. Moriarty, C.~Rebbi, and E.~Vicari.
\newblock Multigrid propagators in the presence of disordered {U(1)} gauge
  fields.
\newblock {\em Phys. Rev.}, D43:1974--1977, 1991.

\bibitem{Brower:1991xv}
Richard~C. Brower, Robert~G. Edwards, Claudio Rebbi, and Ettore Vicari.
\newblock Projective multigrid for {W}ilson fermions.
\newblock {\em Nucl. Phys.}, B366:689--705, 1991.

\bibitem{Brower:1987dd}
Richard~C. Brower, Eric Myers, Claudio Rebbi, and K.~J.~M. Moriarty.
\newblock The multigrid method for fermion calculations in {Q}uantum
  {C}hromodynamics.
\newblock Print-87-0335 (IAS, PRINCETON).

\bibitem{Brower:1990at}
Richard~C. Brower, Claudio Rebbi, and Ettore Vicari.
\newblock Projective multigrid for propagators in lattice gauge theory.
\newblock {\em Phys. Rev. Lett.}, 66:1263--1266, 1991.

\bibitem{Brower:1991ni}
Richard~C. Brower, Claudio Rebbi, and Ettore Vicari.
\newblock Projective multigrid method for propagators in lattice gauge theory.
\newblock {\em Phys. Rev.}, D43:1965--1973, 1991.

\bibitem{DeGrand:2006zz}
T.~DeGrand and C.~E. Detar.
\newblock {\em Lattice Methods for {Q}uantum {C}hromodynamics}.
\newblock World Scientific, 2006.

\bibitem{DelDebbio:2005qa}
L.~Del~Debbio, Leonardo Giusti, M.~Luscher, R.~Petronzio, and N.~Tantalo.
\newblock Stability of lattice {QCD} simulations and the thermodynamic limit.
\newblock {\em JHEP}, 0602:011, 2006.

\bibitem{MILC}
A.~Bazavov et. al.
\newblock The {MILC} collaboration.
\newblock {\em Rev. Mod. Phys}, 82:1349--1417, 2010.

\bibitem{Rottman_arxiv}
Andreas Frommer, Karsten Kahl, Stefan Krieg, Bj\"{o}rn Leder, and Matthias
  Rottmann.
\newblock Aggregation-based multilevel methods for lattice {QCD}.
\newblock {\em pre-print}, arxiv:1202.2462v1, 2012.
\newblock submitted.

\bibitem{Gattringer:2010zz}
C.~Gattringer and C.~B. Lang.
\newblock Quantum {C}hromodynamics on the lattice.
\newblock {\em Lect. Notes Phys.}, 788:1--211, 2010.

\bibitem{golub_kahan_1965}
G.~Golub and W.~Kahan.
\newblock Calculating the singular values and pseudo-inverse of a matrix.
\newblock {\em Journal of the Society for Industrial and Applied Mathematics
  Series B Numerical Analysis}, 2(2):205--224, 1965.

\bibitem{Harmatz1991102}
M.~Harmatz, P.~G. Lauwers, R.~Ben-Av, A.~Brandt, E.~Katznelson, S.~Solomon, and
  K.~Wolowesky.
\newblock Parallel-transported multigrid and its application to the schwinger
  model.
\newblock {\em Nuclear Physics B - Proceedings Supplements}, 20(0):102 -- 109,
  1991.

\bibitem{local}
Tamas~G. Kovacs, Ferenc Pittler, Falk Bruckmann, and Sebastian Schierenberg.
\newblock High temperature quark localization by {P}olyakov loops.
\newblock {\em PoS}, LATTICE2011:200, 2011.

\bibitem{lanczos_1961}
C.~Lanczos.
\newblock {\em Linear Differential Operators}.
\newblock van Nostrand, 1961.

\bibitem{MBrezina_RFalgout_SMacLachlan_TManteuffel_SMcCormick_JRuge_2003}
R.~Falgout M.~Brezina, S.~MacLachlan, T.~Manteuffel, S.~McCormick, and J.~Ruge.
\newblock Adaptive smoothed aggregation ($\alpha${SA}).
\newblock {\em SIAM J. Sci. Comput.}, 25(6):1896--1920, 2004.

\bibitem{montvay1994quantum}
I.~Montvay and G.~M{\"u}nster.
\newblock {\em Quantum Fields on a Lattice}.
\newblock Cambridge Monographs on Mathematical Physics Series. Cambridge
  University Press, 1994.

\bibitem{Osborne_POS_2010}
J.~C. Osborn, R.~Babich, J.~Brannick, R.~C. Brower, M.~A. Clark, et~al.
\newblock Multigrid solver for clover fermions.
\newblock {\em PoS}, LATTICE2010:037, 2010.

\bibitem{IOPORT.05602071}
Constantin Popa.
\newblock Algebraic multigrid smoothing property of {K}aczmarz's relaxation for
  general rectangular linear systems.
\newblock {\em Electron. Trans. Numer. Anal.}, 29:150--162, electronic only,
  2007.

\bibitem{Sala:2008:NPS:1461600.1461602}
Marzio Sala and Raymond~S. Tuminaro.
\newblock A new {P}etrov-{G}alerkin smoothed aggregation preconditioner for
  nonsymmetric linear systems.
\newblock {\em SIAM J. Sci. Comput.}, 31(1):143--166, October 2008.

\bibitem{hans_arxiv}
H.~Sterck.
\newblock A self-learning algebraic multigrid method for extremal singular
  triplets and eigenpairs.
\newblock {\em SIAM J. Sci. Comput.}, 34(4):A2092--A2117, 2012.

\bibitem{Trottenberg:2000:MUL:374106}
Ulrich Trottenberg and Anton Schuller.
\newblock {\em Multigrid}.
\newblock Academic Press, Inc., Orlando, FL, USA, 2001.

\bibitem{Vassilevski_2005}
P.~Vassilevski.
\newblock {\em Multilevel Block Factorization Preconditioners: Matrix-based
  Analysis and Algorithms for Solving Finite Element Equations}.
\newblock Springer, 2009.

\bibitem{PhysRevD.10.2445}
Kenneth~G. Wilson.
\newblock Confinement of quarks.
\newblock {\em Phys. Rev. D}, 10:2445--2459, Oct 1974.

\end{thebibliography}
\bibliographystyle{plain}

\end{document}